\documentclass[a4paper,12pt]{amsart}
\usepackage{amsfonts}
\usepackage{amsmath,array}
\usepackage{amssymb}
\usepackage{mathrsfs}
\usepackage[colorlinks]{hyperref}
 \usepackage{graphicx}
\usepackage{enumerate}
\usepackage[usenames]{color}

\newtheorem{thm}{Theorem}[section]
\newtheorem{cor}[thm]{Corollary}
\newtheorem{lem}[thm]{Lemma}
\newtheorem{prop}[thm]{Proposition}

\numberwithin{equation}{section}
\usepackage[english]{babel} 
\usepackage{blindtext}

\theoremstyle{definition}
\newtheorem{definition}[thm]{Definition}
\newtheorem{rem}[thm]{Remark}

\begin{document}
 \title[Fujita exponents on quantum Euclidean spaces]{Fujita exponents on quantum Euclidean spaces}

\author[E. McDonald]{E. McDonald}
\address{
 Edward McDonald:
  \endgraf
 Laboratoire d’analyse et de math\'ematiques appliqu\'ees, Universit\'e Paris-Est Cr\'eteil
 \endgraf
 France
  \endgraf
    {\it E-mail address} {\rm eamcd92@gmail.com}
  }

\author[M. Ruzhansky]{M. Ruzhansky}
\address{
 Michael Ruzhansky:
  \endgraf
 Department of Mathematics: Analysis, Logic and Discrete Mathematics,
  \endgraf
 Ghent University, Ghent,
 \endgraf
  Belgium
  \endgraf
  and
 \endgraf
 School of Mathematical Sciences, Queen Mary University of London, London,
 \endgraf
 UK
 \endgraf
  {\it E-mail address} {\rm michael.ruzhansky@ugent.be}
  }

\author[S. Shaimardan]{S. Shaimardan}
\address{
  Serikbol Shaimardan:
  \endgraf
  Institute of Mathematics and Mathematical Modeling, 050010, Almaty,
  \endgraf
  Kazakhstan
  \endgraf
  {\it E-mail address} {\rm shaimardan.serik@gmail.com}
  }

 \author[K. Tulenov]{K. Tulenov}
\address{
  Kanat Tulenov:
   \endgraf
      \endgraf
  Institute of Mathematics and Mathematical Modeling, 050010, Almaty,
  \endgraf
  Kazakhstan
    \endgraf
  and
      \endgraf
    School of Mathematics and Statistics, University of New South Wales, Kensington, NSW, 2052,
  \endgraf
 Australia
  \endgraf
  {\it E-mail address} {\rm tulenov@math.kz}
  }

\date{}

\begin{abstract}
We study the well-posedness of non-linear heat equation with power nonlinearity with positive initial data on quantum Euclidean spaces. We prove a noncommutative analogue of the classical Fujita theorem by identifying the critical exponent separating finite-time blow-up from global existence for small initial data. Moreover, we establish a fundamental inequality in general semifinite von Neumann algebras that is of independent interest and plays a crucial role in the study of global existence and local well-posedness of solutions of nonlinear equations in noncommutative setting.
\end{abstract}



\maketitle

\tableofcontents
{\section{Introduction}}
The semilinear heat equation
\begin{equation}\label{eq:heat}
    \partial_t u + \Delta u = u^p,
    \qquad (t,x)\in (0,\infty)\times\mathbb{R}^d,\,d\geq2,\quad u(0)=u_0\ge0,
\end{equation}
where $1<p<\infty$ and $\Delta$ is the (positive) Laplacian in $\mathbb{R}^d,$
is a fundamental model for nonlinear diffusion, capturing the competition between the smoothing effect of the heat operator and the amplifying effect of the nonlinear term.  A classical result of Fujita \cite{Fujita1966} identifies the critical exponent (see, \cite[Theorem 18.1. p.113]{QS})
$$
p_F := 1 + \frac{2}{d}
$$
which separates finite-time blow-up when $1<p\le p_F$ from global existence for small initial data when $p>p_F$. The Fujita exponenent $p_F$ is related  to the self-similarity of the equation \eqref{eq:heat} under rescaling.

Following Fujita, many authors refined and extended the classical theory. Hayakawa \cite{Hayakawa1973} and Weissler \cite{Weissler1981} studied global existence and blow-up in critical and supercritical regimes; Galaktionov and Vázquez \cite{GalaktionovVazquez1995} developed a systematic analysis of blow-up; and Levine \cite{Levine1990} clarified the significance of critical exponents in nonlinear PDEs. A comprehensive account is given in the monograph of Quittner and Souplet \cite{QS}. Further extensions include more general nonlinearities, weighted problems, gradient terms, and equations on non-Euclidean geometries \cite{HerreroVelazquez1993, GigaKohn1985, Levine1990, QS}, with applications ranging from population dynamics to chemical reactions and nonlinear thermal processes.
More recently, the scope of Fujita-type results has broadened to non-Euclidean geometries. For example, in \cite{BRT}, \cite{FRT} authors studied the Fujita exponent for degenerate parabolic equations with power-law nonlinearities and semilinear heat equations on the Heisenberg group. The work \cite{RY} generalized Weissler’s classical result to sub-Riemannian manifolds and unimodular Lie groups. Similarly, the paper \cite{Pa98} investigated analogous problems on nilpotent groups.

The purpose of the present paper is to investigate Fujita exponents in quantum Euclidean spaces. Noncommutative or quantum  Euclidean space $\mathbb{R}^d_{\theta},$ defined by an arbitrary antisymmetric real $d\times d$ matrix $\theta$, are also known as Moyal spaces.  Heuristically, $\mathbb{R}^d_\theta$ is a version of Euclidean space $\mathbb{R}^d$ where the coordinate functions $x_1,\ldots,x_d$ do not commute but instead obey the relation
$$
    x_jx_k-x_kx_j = i\theta_{j,k},\quad 1\leq j,k\leq d.
$$
In accordance with the philosophy of noncommutative geometry, the space $\mathbb{R}^d_\theta$ is studied via spaces of functions such as the algebra of essentially bounded functions $L^{\infty}(\mathbb{R}^d_\theta)$ and the Schwartz space $\mathcal{S}(\mathbb{R}^d_\theta).$ Quantum Euclidean spaces are fundamental examples in noncommutative geometry and serve as prototypical examples of noncompact spaces \cite{CGRS, GGSVM, green-book}.
Noncommutative Euclidean spaces have been defined by many authors, for example Rieffel \cite{Rieffel} introduced them as deformation quantisations of $\mathbb{R}^d$, while related constructions were earlier developed by Moyal \cite{M} and Groenewold \cite{Gro} in the physics literature. The idea of these constructions is to deform the algebra of smooth functions on $\mathbb{R}^d$ by replacing pointwise multiplication with the twisted Moyal product. In this paper we adopt the von Neumann algebra model $L^{\infty}(\mathbb{R}^d_{\theta})$ defined through a twisted left-regular representation of $\mathbb{R}^d$ on $L^2(\mathbb{R}^d)$.

Many analogues of tools from classical harmonic analysis tools have been adapted to $L^{\infty}(\mathbb{R}^d_\theta)$ \cite{GJM, GGSVM, GJP, HLW, MSX, Mc, RST}, including differential operators, Fourier multipliers, and function spaces. This makes it possible to study partial differential equations in a noncommutative space.

The noncommutative measure space $(L^{\infty}(\mathbb{R}^d_{\theta}),\tau_{\theta})$ is equipped with a semifinite normal trace $\tau_{\theta}$ on the von Neumann algebra $L^{\infty}(\mathbb{R}^d_{\theta}),$ which generalizes the Lebesgue integral (recovering it when $\theta = 0$).  The general theory of  noncommutative $L^p$-spaces associated with semifinite von Neumann algebras \cite{DPS, LSZ, PXu} allows the construction of noncommutative $L^p$-spaces, denoted $L^p(\mathbb{R}^d_\theta) := L^p(L^{\infty}(\mathbb{R}^d_\theta), \tau_{\theta})$. When $\theta = 0$, these spaces reduce to the classical $L^p$-spaces on $\mathbb{R}^d.$
Differential calculus is based on the translation semigroup $\{T_s\}_{s\in \mathbb{R}^d},$ which is a group of automorphisms of $L^\infty(\mathbb{R}^d_\theta)$ which we recall in Definition \ref{Definition_1} below.
The Laplace operator $\Delta_{\theta}$ on $L^{\infty}(\mathbb{R}^d_{\theta})$ by the formula
$$
\Delta_{\theta} = -(\partial_1^{\theta})^2   -\cdots -(\partial_d^{\theta})^2,
$$
where the partial derivatives $\partial_i^{\theta}, i=1,2,\cdots,d$   are defined on smooth elements $x$ (see, Definition \ref{Definition_2}) by
$$
\partial^{\theta}_j(x) = \frac{d}{ds_{j}}T_s(x)|_{s=0},
$$

Significant progress has already been made in nonlinear PDEs on noncommutative spaces. Rosenberg \cite{R} analyzed nonlinear equations on the noncommutative torus, while Zhao \cite{Zhao} established Strichartz estimates for the Schrödinger equation on noncommutative Euclidean spaces. The paper \cite{Mc} developed a paradifferential calculus adapted to this setting and applied it to nonlinear evolution equations. In \cite{RST2}, the authors proved Sobolev and Gagliardo--Nirenberg inequalities on $L^{\infty}(\mathbb{R}^d_{\theta})$, yielding global well-posedness results for nonlinear PDEs. Likewise, \cite{RST1} established $L^p$--$L^q$ bounds for Fourier multipliers on $L^{\infty}(\mathbb{R}^d_{\theta})$ and applied them to heat, wave, and Schr\"odinger-type equations with Caputo fractional derivatives for $1<p\le2\le q<\infty$, obtaining corresponding well-posedness results (see also \cite{OTT}). Similar developments for the noncommutative torus appear in \cite{STT}. The most recent developments concern Navier-Stokes equations \cite{CHWW}, and Besov spaces \cite{ChenHong}.
One motivation for studying PDE on quantum Euclidean spaces is that the noncommutativity results in some counterintuitive phenomena. For example when $\det(\theta)\neq 0,$ the Banach space $L^p(\mathbb{R}^d_\theta)$ is an algebra, and we have the nesting
$$
    L^p(\mathbb{R}^d_\theta)\subset L^q(\mathbb{R}^d_\theta),\quad p<q.
$$
Another feature, which motivated the present work, is that the power function $u\mapsto u^p$ on positive operators is not monotone when $p>1,$ and is not convex when $p>2.$

This paper is devoted to the study of the following analogy of \eqref{eq:heat}:
\begin{equation} \label{eq:nc-heat_problem}
\begin{aligned}
    \frac{\partial}{\partial t}u(t) &= -\Delta_{\theta} u(t) + u^p(t), \quad p>1,\; t>0,\\
    u(0) &= u_0,
\end{aligned}
\end{equation}
where $u_0\in L^{\infty}(\mathbb{R}^d_{\theta})$ is non-negative.  Our motivation for studying \eqref{eq:nc-heat_problem} is twofold:
\begin{enumerate}
    \item Fujita's original argument \cite{Fujita1966}, as well as later developments, rely heavily on the monotonicity of the function $u\mapsto u^p.$ However, when $p>1,$ this map is not monotone on positive operators.
    \item{} The critical exponent $p_F$ can be identified by the self-similarity of \eqref{eq:heat} under rescaling. However, $L^{\infty}(\mathbb{R}^d_{\theta})$ does not have a scaling symmetry. Should we expect the same critical exponent?
\end{enumerate}
 The main aim of this paper is to address the above two fundamental obstacles and to establish a noncommutative analogue of Fujita’s classical result, which recovers  Fujita’s theorem in the commutative case $\theta=0.$ We employ methods from noncommutative analysis such as double operator integrals \cite{BirmanKaradzhovSolomyak}.

A key step in establishing local well-posedness is the inequality for any non-negative $u,v\in L^{pq}(\mathbb{R}^d_\theta)$ with  $1\le q \le\infty:$
\begin{equation}\label{Non_lin_0}
\|u^p - v^p\|_{L^{q}(\mathbb{R}^d_\theta)}
\le C_p \|u^{p-1}(u-v)+(u-v)v^{p-1}\|_{L^{q}(\mathbb{R}^d_\theta)},\, 1\leq p<\infty,
\end{equation}
where $C_p>0$ is a constant. Note that this inequality is very easily obtained in the commutative case, but is not at all trivial in the noncommutative case.

We establish \eqref{Non_lin_0} in the setting of general semifinite von Neumann algebras since we believe that it is of independent interest.  The estimate \eqref{Non_lin_0} plays a crucial role in the analysis of nonlinear partial differential equations, particularly in the study of global existence and local well-posedness of solutions of PDEs with power-type nonlinearities. Using this inequality, we can extend the main results of \cite[Theorem 5.4. p.31]{RST2}, \cite[Theorem 4.2. p.85]{RST1} and \cite[Theorem 5.1.2. p.16]{STT} (see Remark \ref{Exp-remark} below) to the full range
of $1\leq p<\infty$, which were previously known only for positive integer values of $p.$

It is likely that many of the results here can be generalised from $L^{\infty}(\mathbb{R}^d_\theta)$ to other semifinite von Neumann algebras equipped with a trace-preserving semigroup. We leave this generalisation to future work.

\bigskip

\section{Preliminaries}

 \subsection{Noncommutative (NC) Euclidean space $L^{\infty}(\mathbb{R}^{d}_{\theta})$} \label{NC Euclidean space}
We refer the reader to \cite{GJM}, \cite{GGSVM}, \cite{MSX}, \cite{Mc}, and \cite{LSMcZ} for a detailed discussion of the noncommutative (NC) Euclidean space $L^{\infty}(\mathbb{R}^d_{\theta})$ and for additional information. Throughout this paper, we assume that $d\geq 2.$

Let $H$ be a Hilbert space. We denote by $B(H)$ the algebra of all bounded linear operators on $H.$ As usual $L^p(\mathbb{R}^d)$ ($1\leq p<\infty$) are the $L^p$-spaces of pointwise almost-everywhere equivalence classes of $p$-integrable functions and $L^{\infty}(\mathbb{R}^d)$ is the space of essentially bounded functions on the Euclidean space $\mathbb{R}^d.$
Given an integer $d\geq 2,$ fix an anti-symmetric $\mathbb{R}$-valued $d\times d$ matrix $\theta=\{\theta_{j,k}\}_{1\leq j,k\leq d}.$

\begin{definition} \cite[Definition 2.1]{Mc}\label{NC_E_space2} For $t\in\mathbb{R}^d,$ denote by $U_{\theta}(t)$ the operator on $L^{2}(\mathbb{R}^{d})$ defined by the formula
$$
(U_{\theta}(t)\xi)(s)=e^{i(t,s)}\xi(s-\frac{1}{2}\theta t),\quad \xi\in L^{2}(\mathbb{R}^{d}),  \, t,s\in \mathbb{R}^d
$$
where $(\cdot,\cdot)$ means the usual inner product in $\mathbb{R}^d.$
\end{definition}
It can be proved that the family $\{U_{\theta}(t)\}_{t\in \mathbb{R}^{d}}$ is strongly continuous and satisfies
\begin{equation}\label{weyl-relation}
U_{\theta}(t)U_{\theta}(s)=e^{\frac{1}{2}i(t,\theta s)}U_{\theta}(t+s),\quad t,s\in \mathbb{R}^d,
\end{equation}
The above relation is called the Weyl form of the canonical commutation relation.

The von Neumann algebra $L^{\infty}(\mathbb{R}_{\theta}^d)$ is defined to be the weak operator topology closed subalgebra of $B(L^{2}(\mathbb{R}^{d}))$ generated by the family $\{U_{\theta}(t)\}_{t\in \mathbb{R}^{d}}.$
\begin{definition}\label{NC_E_space1}
Define $L^{\infty}(\mathbb{R}^{d}_{\theta})$ as the von Neumann algebra generated by the $d$-parameter strongly continuous unitary family $\{U_{\theta}(t)\}_{t\in \mathbb{R}^{d}}$ on $L^{2}(\mathbb{R}^d)$ from Definition \ref{NC_E_space2}.
\end{definition}
It is also possible to define $L^{\infty}(\mathbb{R}_{\theta}^d)$ in an abstract operator-theoretic way more reminiscent of the conventional approach for quantum tori, see \cite{GGSVM}.

Note that if $\theta=\mathbf{0},$ where $\mathbf{0}$ is the zero matrix, then above definitions reduce to the description of $L^{\infty}(\mathbb{R}^{d})$ as the algebra of bounded pointwise multipliers on $L^{2}(\mathbb{R}^{d}).$ The structure of $L^{\infty}(\mathbb{R}_{\theta}^d)$ is determined by the Stone-von Neumann theorem, which implies that if $\det(\theta)\neq 0,$ then any two $C^*$-algebras generated by a strongly continuous unitary family $\{U_{\theta}(t)\}_{t\in \mathbb{R}^{d}}$ satisfying the Weyl relation \eqref{weyl-relation} are $*$-isomorphic \cite[Theorem 14.8]{H}.

Generally, the algebraic nature of the noncommutative Euclidean space $L^{\infty}(\mathbb{R}_{\theta}^d)$ depends on the dimension of the kernel of $\theta.$ In the case $d=2,$ up to an orthogonal conjugation $\theta$ may be given as
\begin{equation}\label{d=2}
    \theta=h\begin{pmatrix}
0 & -1\\
1 & 0
\end{pmatrix}
\end{equation}
for some constant $h>0.$ In this case, $L^{\infty}(\mathbb{R}_{\theta}^2)$ is $*$-isomorphic to $B(L^{2}(\mathbb{R}))$ and this $*$-isomorphism can be written as
$$U_{\theta}(t)\to e^{it_1\mathbf{M}_s +t_2 h \frac{d}{ds}},$$
where $\mathbf{M}_s\xi(s)=s\xi(s)$ and $\frac{d}{ds}\xi(s)=\xi'(s)$ is the differentiation operator.
If $d\geq 2,$ then an arbitrary $d\times d$ antisymmetric real matrix can be expressed (up to orthogonal conjugation) as a direct sum of a zero matrix and matrices of the form \eqref{d=2}, ultimately leading to the $*$-isomorphism
\begin{equation}\label{direct-sum}
    L^{\infty}(\mathbb{R}_{\theta}^d)\cong L^{\infty}(\mathbb{R}^{\dim(\ker(\theta))})\bar{\otimes}B(L^{2}(\mathbb{R}^{\text{rank}(\theta)/2})),
\end{equation}
where $\bar{\otimes}$ is the von Neumann algebra tensor product \cite[Corollary 6.4]{LeSZ}. In particular, if $\det(\theta)\neq 0,$ then \eqref{direct-sum} reduces to
\begin{equation}\label{reduced-direct-sum}
    L^{\infty}(\mathbb{R}_{\theta}^d)\cong B(L^{2}(\mathbb{R}^{d/2})).
\end{equation}
Note that these formulas are meaningful since the rank of an anti-symmetric matrix is always even.
\subsection{Noncommutative integration} Let $f\in L^{1}(\mathbb{R}^d).$ We define
\begin{equation}\label{def-integration}
\lambda_{\theta}(f)\xi=\int_{\mathbb{R}^d}f(t)U_{\theta}(t)\xi dt, \quad \xi\in L^{2}(\mathbb{R}^d).
\end{equation}
This integral is absolutely convergent in the $L^2(\mathbb{R}^d)$-valued Bochner sense, and defines a bounded linear operator $\lambda_{\theta}(f): L^{2}(\mathbb{R}^d)\to L^{2}(\mathbb{R}^d)$ such that $\lambda_{\theta}(f)\in L^{\infty}(\mathbb{R}_{\theta}^d)$ (see, \cite[Lemma 2.3]{MSX}).  Denote by $\mathcal{S}(\mathbb{R}^d)$ the Schwartz space on $\mathbb{R}^d.$  The noncommutative Schwartz space $\mathcal{S}(\mathbb{R}_{\theta}^d)$ is defined as the image of the classical Schwartz space under $\lambda_{\theta},$ which is
\begin{equation}\label{NC_Schwartz}\mathcal{S}(\mathbb{R}_{\theta}^d):=\{x\in L^{\infty}(\mathbb{R}_{\theta}^d):x=\lambda_{\theta}(f) \,\ \text{for some}\,\ f\in \mathcal{S}(\mathbb{R}^d)\}.
\end{equation}
We define a topology on $\mathcal{S}(\mathbb{R}_{\theta}^d)$
 as the image of the canonical Fr\'{e}chet topology on $\mathcal{S}(\mathbb{R}^d)$ under $\lambda_{\theta}.$ The topological dual of $\mathcal{S}(\mathbb{R}_{\theta}^d)$ will be denoted by $\mathcal{S}'(\mathbb{R}_{\theta}^d).$ The mapping  $\lambda_\theta:\mathcal{S}(\mathbb{R}^d)\to \mathcal{S}(\mathbb{R}^d_\theta)$ extends to a bijection from
$\mathcal{S}'(\mathbb{R}^d)$ to $\mathcal{S}'(\mathbb{R}^d_\theta)$ (see \cite[Subsection 2.2.3]{MSX}). For $f\in\mathcal{S}(\mathbb{R}^d)$ we normalise the Fourier transform as
$$\widehat{f}(t)=\int_{\mathbb{R}^d}f(s)e^{-i(t,s)}ds, \quad t\in \mathbb{R}^d.$$

\subsection{Trace on $L^{\infty}(\mathbb{R}_{\theta}^d)$} Given $f\in \mathcal{S}(\mathbb{R}^d),$ we
define the functional $\tau_{\theta}:\mathcal{S}(\mathbb{R}_{\theta}^d)\to \mathbb{C}$ by the formula
\begin{equation}\label{def_trace}
\tau_{\theta}(\lambda_{\theta}(f)):=(2\pi)^df(0).
\end{equation}
Define the bilinear and sesquilinear pairings
$$
    (u,v) := \tau_\theta(uv),\quad \langle u,v\rangle = \tau_\theta(u^*v).
$$
We have the Plancherel identities
$f \in \mathcal{S}'(\mathbb{R}^d)$,
\begin{equation} \label{Main_eqlation1}
\bigl( \lambda_\theta(f), \lambda_\theta(g) \bigr) = (f, \widetilde{g}), \qquad
\text{for all } g \in \mathcal{S}(\mathbb{R}^d).
\end{equation}
where $\widetilde{g}(t) = g(-t)$ and
\begin{equation}\label{plancherel}
    \langle \lambda_\theta(f),\lambda_\theta(g)\rangle = \langle f,g\rangle_{L^2(\mathbb{R}^d)}.
\end{equation}

We write the bilinear pairing of $\mathcal{S}'(\mathbb{R}^d_\theta)$ and $\mathcal{S}(\mathbb{R}^d_\theta)$ by $(\cdot,\cdot)$ in order that \eqref{Main_eqlation1} also holds for $f \in \mathcal{S}'(\mathbb{R}^d)$ and $g \in \mathcal{S}(\mathbb{R}^d).$

The functional $\tau_{\theta}$ admits an extension to a semifinite faithful normal trace on $L^{\infty}(\mathbb{R}_{\theta}^d).$ Moreover, if $\theta=\mathbf{0},$   then $\tau_{\theta}$ can be identified with the Lebesgue integral. If $\det(\theta)\neq0,$
then $\tau_{\theta}$ can be (up to a normalisation) identified with the operator trace on $B(L^2(\mathbb{R}^{d/2})).$ In the general case $\tau_\theta$ can be identified with a tensor product of a Lebesgue integral and an operator trace. For more details, we refer to \cite{GJP}, \cite[Lemma 2.7]{MSX}, \cite[Theorem 2.6]{Mc}.

\subsection{Noncommutative $L^{p}(\mathbb{R}^{d}_{\theta})$ spaces}
With the definitions in the previous sections,  $L^{\infty}(\mathbb{R}^{d}_{\theta})$ is a semifinite von Neumann algebra with the trace $\tau_{\theta},$ and the pair $(L^{\infty}(\mathbb{R}^{d}_{\theta}),\tau_{\theta})$ is called a noncommutative measure space. For any $1\leq p<\infty,$ we can define the $L^p$-norm on this space by the Borel functional calculus and the following formula:
$$\|x\|_{L^p(\mathbb{R}^d_{\theta})}=\Big(\tau_{\theta}(|x|^p)\Big)^{1/p},\quad x\in L^{\infty}(\mathbb{R}^{d}_{\theta}), $$
where $|x|:=(x^{*}x)^{1/2}.$
The completion of $\{x\in L^{\infty}(\mathbb{R}^{d}_{\theta}) \;:\;\|x\|_{p}<\infty\}$ with respect to $\|\cdot\|_{L^p(\mathbb{R}^d_{\theta})}$ is denoted by $L^p(\mathbb{R}^d_{\theta}).$ The elements of $L^p(\mathbb{R}^d_{\theta})$ are $\tau_{\theta}$-measurable operators like in the commutative case. These are linear densely defined closed (possibly unbounded) affiliated with $L^{\infty}(\mathbb{R}^{d}_{\theta})$ operators such that $\tau_{\theta}(\mathbf{1}_{(s,\infty)}(|x|))<\infty$ for some $s>0.$ Here, $\mathbf{1}_{(s,\infty)}(|x|)$ is the spectral projection of $|x|$ with respect to the interval $(s,\infty).$

The Schwartz space $\mathcal{S}(\mathbb{R}^d_\theta)$ is dense in $L^p(\mathbb{R}^d_\theta)$ for all $p<\infty,$ see \cite[Proposition 3.14]{MSX}.
The identity \eqref{plancherel} implies that $\lambda_\theta$ extends to an isometric isomorphism
$$
    \lambda_\theta:L^2(\mathbb{R}^d)\to L^2(\mathbb{R}^d_\theta).
$$
Let us denote the set of all
$\tau_{\theta}$-measurable operators by $L^{0}(\mathbb{R}^{d}_{\theta}).$
Let $x=x^{\ast}\in L^{0}(\mathbb{R}^{d}_{\theta})$.  The {\it distribution function} of $x$ is defined by
$$n_x(s)=\tau_{\theta}\left(\mathbf{1}_{(s,\infty)}(x)\right), \quad -\infty<s<\infty.$$
For $x\in L^{0}(\mathbb{R}^{d}_{\theta}),$ the {\it generalised singular value function} $\mu(t, x)$ of $x$ is defined by
\begin{equation}\label{distribution-function}
\mu(t,x)=\inf\left\{s>0: n_{|x|}(s)\leq t\right\}, \quad t>0.
\end{equation}
The function $t\mapsto\mu(t,x)$ is decreasing and right-continuous. For more discussion on generalised singular value functions, we refer the reader to
\cite{DPS, FK, LSZ}.
The norm of $L^p(\mathbb{R}^d_{\theta})$ can also be written in terms of the generalised singular value function (see \cite[Example 2.4.2, p. 53]{LSZ}) as follows
\begin{equation}\label{mu-norm}
\|x\|_{L^p(\mathbb{R}^d_{\theta})}=\left(\int_{0}^{\infty}\mu^{p}(s,x)ds\right)^{1/p}, \,\ \text{for} \,\ p<\infty\,\ \text{and} \,\,
\|x\|_{L^{\infty}(\mathbb{R}^d_{\theta})}=\mu(0,x), \,\, \text{for}\,\ p=\infty.
\end{equation}
The latter equality for $p=\infty,$ was proved in \cite[Lemma 2.3.12. (b), p. 50]{LSZ}.

The space $L^{0}(\mathbb{R}^{d}_{\theta})$ is a $*$-algebra, which can be made into a topological $*$-algebra as follows. Let
$$V(\varepsilon,\delta)=\{x\in L^{0}(\mathbb{R}^{d}_{\theta}): \mu(\varepsilon,x)\leq \delta\}.$$
Then $\{V(\varepsilon,\delta): \varepsilon,\delta>0\}$ is a system of neighbourhoods at 0 for which $L^{0}(\mathbb{R}^{d}_{\theta})$ becomes a metrizable topological $*$-algebra. The convergence with respect to this topology is called the {\it convergence in measure} \cite{PXu}. For the theory of $L^p$ spaces corresponding to general semifinite von Neumann algebras, we refer the reader to  \cite{DPS}, \cite{LSZ}, \cite{PXu}.

\subsection{Differential calculus on $L^{\infty}(\mathbb{R}^{d}_{\theta})$}  We embed the space $L^1(\mathbb{R}^d_{\theta})+ L^\infty(\mathbb{R}^d_{\theta})$  into $\mathcal{S}'(\mathbb{R}_{\theta}^d)$ via
\begin{equation}\label{Main_relation2}
(u, v):=\tau_{\theta}(uv),\quad u\in L^1(\mathbb{R}^d_{\theta})+ L^\infty(\mathbb{R}^d_{\theta}), \quad v\in \mathcal{S}(\mathbb{R}_{\theta}^d).
\end{equation}

\begin{definition}\label{Definition_1} For $t \in\mathbb{R}^d,$ define $T_t$ as
the automorphism of $L^\infty(\mathbb{R}^d_\theta)$ specified by
$$
    T_t(U_\theta(s)) = \exp(i(t,s))U_{\theta}(s),\quad s \in \mathbb{R}^d.
$$
More generally, if $x\in\mathcal{S}'(\mathbb{R}_{\theta}^d),$ define $T_t( x )$ as the distribution given by
$$
(T_t (x), y) = ( x, T_{-t}(y)), \quad y\in\mathcal{S}(\mathbb{R}^d).
$$
\end{definition}
That $T_t(x)$ is a well-defined distribution for all
 $x\in\mathcal{S}'(\mathbb{R}_{\theta}^d)$ is a straightforward consequence of the observation that $T_t$ is continuous in every seminorm which defines
the topology of $\mathcal{S}(\mathbb{R}_{\theta}^d).$ Moreover, it is a trivial matter to verify that $T_t$ is an isometry in
every $L^p(\mathbb{R}^d_{\theta})$ for $0 < p \le \infty.$ In terms of the map $\lambda_{\theta},$ we have:
\begin{equation}\label{U-property}
T_t\lambda_{\theta}(f) = \lambda_{\theta}(e^{i(t,\cdot)}f(\cdot)),\quad f \in \mathcal{S}(\mathbb{R}^d).
\end{equation}

\begin{definition}\label{Definition_2} (\cite[Definition 2.9]{Mc}) An element $x \in L^{1}(\mathbb{R}^d_\theta) + L^{\infty}(\mathbb{R}^d_\theta)$   is said to be smooth if for all $y \in L^{1}(\mathbb{R}^d_\theta) \cap L^{\infty}(\mathbb{R}^d_\theta)$ the function $s \mapsto \tau_{\theta}(yT_s(x))$ is smooth.
\end{definition}

The partial derivatives $\partial^{\theta}_j,  j = 1, \dots, d,$ are defined on smooth elements $x$ by
$$
\partial^{\theta}_j(x) = \frac{d}{ds_{j}}T_s(x)|_{s=0}.
$$
From \eqref{U-property} and \eqref{def-integration} it is easily verified that
\begin{equation}\label{partial_derivative}
\partial^{\theta}_j(x)\overset{\eqref{def-integration}}{=}\partial^{\theta}_j\lambda_{\theta}(f)\overset{\eqref{U-property}}{=}\lambda_{\theta}(it_{j}f(t)), \quad  j=1,\cdots,d, \quad f\in \mathcal{S}(\mathbb{R}^d),
\end{equation}
for $x=\lambda_{\theta}(f).$  We can also define $\partial_j^\theta$ on distributions using \eqref{partial_derivative}.

For a multi-index $\alpha=(\alpha_1,...,\alpha_d),$ we define
$$
\partial^{\alpha}_{\theta}=(\partial^{\theta}_{1})^{\alpha_{1}}\dots(\partial^{\theta}_{d})^{\alpha_{d}},
$$
and the gradient $\nabla_{\theta}$ associated with $L^{\infty}(\mathbb{R}_{\theta}^d)$ is the operator
$$
\nabla_{\theta}=(\partial^{\theta}_{1},  \dots,  \partial^{\theta}_{d}).
$$
Moreover, the Laplace operator $\Delta_{\theta}$ is defined as
\begin{equation}\label{laplacian}
\Delta_{\theta} = -(\partial_1^{\theta})^2   - \cdots -(\partial_d^{\theta})^2,
\end{equation}
in order that $\Delta_{\theta}$ is a positive operator  on $L^2(\mathbb{R}^{d}_\theta)$ (see  \cite {MSX} and \cite{Mc}).

Now, let us recall the differential structure on $L^{\infty}(\mathbb{R}^d_{\theta})$ (see, \cite[Subsection 2, p. 10]{Mc}).

\begin{definition}\label{F-transform}
For any $x \in \mathcal{S}(\mathbb{R}_{\theta}^d),$ we define the Fourier transform of $x$ as the map $\lambda_{\theta}^{-1}:\mathcal{S}(\mathbb{R}_{\theta}^d)\to \mathcal{S}(\mathbb{R}^d)$ by the formula
\begin{equation}\label{direct-F-transform}
\lambda_{\theta}^{-1}(x):=\widehat{x}, \quad \widehat{x}(s)=\tau_{\theta}(xU_{\theta}(s)^*), \,\ s\in \mathbb{R}^d.
\end{equation}
\end{definition}
As it was already introduced in \cite[Section 3.2]{MSX}, \cite[Section 2.4]{M}, we define the convolution on  $\mathbb{R}^d_{\theta}.$
\begin{definition} Let $1\leq p \leq \infty$ and $x \in L^p(\mathbb{R}^d_{\theta
}).$ For $K\in L^1(\mathbb{R}^{d}),$ we define
\begin{eqnarray}\label{def-convolution}
K\ast x=\int\limits_{\mathbb{R}^{d}}K(t)T_{-t}(x)dt,
\end{eqnarray}
where the integral is understood in the sense of a $L^p(\mathbb{R}^{d})$ - valued Bochner integral when $ p < \infty.$ For $p=\infty,$ \eqref{def-convolution} is understood as a weak$^*$ integral meaning that for all $v \in L^1(\mathbb{R}^d_\theta)$ we have
$$
    \tau_\theta(v(K\ast x)) = \int_{\mathbb{R}^d} K_t(t)\tau_\theta(vT_{-t}(x))\,dt.
$$
\end{definition}

Let $J\subseteq [0, \infty]$ be an interval For a Banach space $X,$ we denote by $C(J, X)$ the space of continuous $X$-valued  functions on the interval $J.$ If $J$ is compact this is a Banach space with norm
$$
\| f \|_{C(J,X)} = \sup\limits_{ s \in J} \| f(s) \|_X.
$$
 Given  $u \in C(J, X),$ we write   $u(t)$   for the value of  $u$  at time $t\in J.$ For points in the interior of $J,$ We write $\partial_t u $ for the derivative of $u$ in the $X$-valued sense, i.e.,
\begin{equation}\label{def_time_derivative}
\lim_{h \to 0} \left\| \frac{u(t+h) - u(t)}{h} - \partial_t u(t) \right\|_X = 0.
\end{equation}

\section{Heat kernel on quantum Euclidean spaces}

Let $G_t$ be the classical Gaussian function on $\mathbb{R}^d,$ i.e.
\begin{equation}\label{def-Gaussian}
G_t(\eta)=(4\pi{t})^{-\frac{d}{2}} e^{-\frac{|\eta|^2}{4t}}, \quad t>0,\quad \eta\in\mathbb{R}^d.
\end{equation}
The heat flow $u\mapsto \exp(-t\Delta_{\theta})u$ is related to the Gaussian function $G_t$ in the same way as the classical case, as the following lemma shows.

\begin{lem}\label{Lemma_1} For any $u\in\mathcal{S}(\mathbb{R}^d_{\theta}),$ we have
\begin{equation}\label{convolution-reprasantation_1}
e^{-t\Delta_{\theta}}u= G_t\ast u,\quad t>0,
\end{equation}
where $G_t$ is the Gaussian function defined by \eqref{def-Gaussian}.
 \end{lem}
\begin{proof}
Since $f\mapsto U_{\theta}(f)$ is an isometry of Hilbert spaces, we have
\begin{equation}\label{def_heat_kernel}
e^{-t\Delta_{\theta}}u:= (2\pi)^{-d}\int\limits_{\mathbb{R}^d}\widehat{G_t}(\xi)\widehat{u}(\xi)U_{\theta}(\xi)d\xi, \quad u\in\mathcal{S}(\mathbb{R}^d_{\theta}),\quad t>0.
\end{equation}
It follows from \eqref{U-property}, \eqref{def-convolution} and \eqref{def_heat_kernel} that
\begin{eqnarray*}
e^{-t\Delta_\theta}u&\overset{\eqref{def_heat_kernel}}{=}& (2\pi)^{-d}\int\limits_{\mathbb{R}^d} \widehat{G_t}(\xi)\widehat{u}(\xi)U_{\theta}(\xi)d\xi\\
&=& (2\pi)^{-d}\int\limits_{\mathbb{R}^d} \int\limits_{\mathbb{R}^d}G_t(\eta)e^{-i(\xi, \eta)}d\eta\widehat{u}(\xi)U_{\theta}(\xi)d\xi\\
&=&(2\pi)^{-d}\int\limits_{\mathbb{R}^d} G_t(\eta)\int\limits_{\mathbb{R}^d}e^{-i(\xi, \eta)}\widehat{u}(\xi)U_{\theta}(\xi)d\xi d\eta\\
&=&\int\limits_{\mathbb{R}^d} G_t(\eta)(2\pi)^{-d}\lambda_{\theta}(e^{-i(\eta, \cdot)}\widehat{u}) d\eta\\
&\overset{\eqref{U-property}}{=}&\int\limits_{\mathbb{R}^d}G_t(\eta)T_{-\eta}(u) d\eta\\
&\overset{\eqref{def-convolution}}=&G_t\ast u,\quad t>0,
\end{eqnarray*}
thereby completing the proof.
\end{proof}

The following result summarizes the core positivity, contractivity, and smoothing properties of the noncommutative heat semigroup.
\begin{lem}\label{HK_Proposition1}  We have the following properties:
\begin{enumerate}[{\rm (a)}]
    \item \label{heat_is_isometry_on_positive_L_1} If $0\leq u \in L^1(\mathbb{R}^d_{\theta}),$ then $e^{-t\Delta_{\theta}}u\geq0$ and $\|e^{-t\Delta_{\theta}}u\|_{L^1(\mathbb{R}^d_{\theta})}=\|u\|_{L^1(\mathbb{R}^d_{\theta})};$
    \item \label{L_p_L_p_mapping_HK_prop_1} If $1\le p \le \infty,$ then $\|e^{-t\Delta_{\theta}}u\|_{L^p(\mathbb{R}^d_{\theta})}\le\|u\|_{L^p(\mathbb{R}^d_{\theta})},\,\ u\in L^p(\mathbb{R}^d_{\theta});$
    \item\label{L_p_L_q_mapping_HK_prop_1} If $1 \le p < q \le \infty$ and $\frac{1}{r} = \frac{1}{p} - \frac{1}{q},$ then for all $t>0$ we have
    $$\|e^{-t\Delta_{\theta}}u\|_{L^q(\mathbb{R}^d_{\theta})} \le (4\pi{t})^{-\frac{d}{2r}}\|u\|_{L^p(\mathbb{R}^d_{\theta})},\quad u\in L^p(\mathbb{R}^d_{\theta}).$$
\end{enumerate}
\end{lem}
\begin{proof}
First we prove \eqref{heat_is_isometry_on_positive_L_1}. Indeed, if $0\leq u \in L^1(\mathbb{R}^d_{\theta}),$ then $T_{\eta}u\geq 0$ for all $\eta\in \mathbb{R}^d.$ Hence, $e^{-t\Delta_{\theta}}(u)\geq0$ (see \cite[Remark 2.7. p.11]{RST}).  Consequently,
\begin{eqnarray*}
\|e^{-t\Delta_{\theta}}u\|_{L^1(\mathbb{R}^d_{\theta})}&=&\tau_{\theta}(e^{-t\Delta_{\theta}}(u))\\
&\overset{\eqref{def_heat_kernel}}{=}&\tau_{\theta}(\int\limits_{\mathbb{R}^d}\widehat{G_t}(\xi)\widehat{u}(\xi)U_{\theta}(\xi)d\xi)\\
&=&\widehat{G_{t}}(0)\widehat{u}(0)=\widehat{u}(0)\\
&=&\tau_{\theta}(u)=\|u\|_{L^1(\mathbb{R}^d_{\theta})}.
\end{eqnarray*}
Properties \eqref{L_p_L_p_mapping_HK_prop_1} and \eqref{L_p_L_q_mapping_HK_prop_1} were already obtained in \cite[Theorem 3.11, p. 15]{Mc} and \cite[Lemma 3.12, p. 16]{Mc}, respectively. This concludes the proof.
\end{proof}

\subsection{Equivalent form of $L^1$-norm}
This subsection is dedicated to the proof of the following theorem:
\begin{thm}\label{tauberian_type_theorem}
    Let $0\leq u\in L^{\infty}(\mathbb{R}^d_\theta).$ Then $u\in L^1(\mathbb{R}^d_\theta)$ if and only if  $$
    \sup\limits_{t>0}\, (4\pi t)^{\frac{d}{2}}\|G_t\ast u\|_{L^{\infty}(\mathbb{R}^d_\theta)} < \infty
    $$
    and moreover
$$
        \|u\|_{L^1(\mathbb{R}^d_\theta)} \leq \sup\limits_{t>0}\,  (4\pi t)^{\frac{d}{2}}\|G_t\ast u\|_{L^\infty(\mathbb{R}^d_\theta)}.
$$
\end{thm}
Note that the direction
$$
    u\in L^1(\mathbb{R}^d_\theta)\Rightarrow \sup\limits_{t>0} (4\pi t)^{\frac{d}{2}}\,\|G_t\ast u\|_{L^\infty(\mathbb{R}^d_\theta)}<\infty
$$
is an immediate consequence of Lemma \ref{HK_Proposition1} \eqref{L_p_L_q_mapping_HK_prop_1}. Hence we concentrate on the converse direction. In the commutative case, Theorem \ref{tauberian_type_theorem} can be deduced from the Fatou lemma, because
$$
    \|u\|_{L^1(\mathbb{R}^d)} = \int_{\mathbb{R}^d} u(x)\,dx \leq \liminf_{t\to\infty} \int_{\mathbb{R}^d} e^{-\frac{|x|^2}{4t}}u(x)\,dx.
$$
In the noncommutative case, we replace the the Fatou lemma by the Fatou property (see, \cite[Theorem 3.4.17. p.151]{DPS}).
\begin{lem}\label{abelian_type_lemma}
    If $u\in L^1(\mathbb{R}^d_\theta),$ then as
$t \to \infty$,
$$
        (4\pi t)^{\frac{d}{2}}G_t\ast u \to \tau_\theta(u)I
$$
    in the weak$^*$-topology of $L^{\infty}(\mathbb{R}^d_\theta),$ where $I$ denotes the identity operator in $L^{\infty}(\mathbb{R}^d_\theta).$
\end{lem}
\begin{proof}
    First we show the convergence in $\mathcal{S}'(\mathbb{R}^d_\theta).$
    Recall that the $\delta$-distribution is defined on $f\in \mathcal{S}(\mathbb{R}^d)$ by
$$
    (\delta,f) = f(0).
$$
Formally,
\begin{equation}\label{Def_Id_op1}
 I=\int\limits_{\mathbb{R}^d} \delta(\xi)U_{\theta}(\xi)d\xi,
\end{equation}
in the sense that for all $v \in \mathcal{S}(\mathbb{R}^d_\theta),$ we have
$$
    \langle I,v\rangle = \tau(v) = \widehat{v}(0).
$$
Let $v\in \mathcal{S}(\mathbb{R}_{\theta}^d)$ be such that $v=\lambda_{\theta}(f)$ with $f\in \mathcal{S}(\mathbb{R}^d).$  Then,  it follows from \eqref{def_heat_kernel} that
\begin{eqnarray*}
G_t\ast v&\overset{\eqref{def_heat_kernel}}=&  \int\limits_{\mathbb{R}^d}\widehat{G}_t(\eta)f(\eta) U_{\theta}(\eta) d\eta.
\end{eqnarray*}
Moreover,  the distribution  $G_t\ast u\in \mathcal{S}'(\mathbb{R}_{\theta}^d)$  is defined by the formula
\begin{equation}\label{Convolution_Distribution}
\langle G_t\ast u, v\rangle = \langle u, G_t \ast v\rangle,\quad u\in L^1(\mathbb{R}_{\theta}^d), \quad v\in\mathcal{S}(\mathbb{R}_{\theta}^d).
\end{equation}
Then, substituting $\eta=\frac{\xi}{\sqrt{4\pi{t}}}$ and applying the Riemann-Lebesgue type lemma (see, \cite[Lemma 2.5.]{HLW}) we obtain
\begin{eqnarray}\label{G_lim_1}
 \lim\limits_{t\to \infty}(4\pi{t})^{\frac{d}{2}}\langle G_t\ast u, v\rangle &\overset{\eqref{Convolution_Distribution}}=& \lim\limits_{t\to \infty}(4\pi{t})^{\frac{d}{2}}\langle u, G_t\ast v\rangle\nonumber\\
  &\overset{\eqref{Main_relation2}}=&  \lim\limits_{t\to \infty}(4\pi{t})^{\frac{d}{2}}\tau_{\theta}(u G_t\ast v)\nonumber\\
  &=& \lim\limits_{t\to \infty}(4\pi{t})^{\frac{d}{2}}\int\limits_{\mathbb{R}^d}\widehat{G}_t(\eta) \widehat{u}(-\eta)f(\eta)d\eta\\
&\overset{\eta=\frac{\xi}{\sqrt{4\pi{t}}}}=& \lim\limits_{t\to \infty}\int\limits_{\mathbb{R}^d}\widehat{G}_t(\frac{\xi}{\sqrt{4\pi{t}}})\widehat{u}(-\frac{\xi}{\sqrt{4\pi{t}}})f(\frac{\xi}{\sqrt{4\pi{t}}})d\xi\nonumber\\
&=&\widehat{u}(0)f(0).\nonumber
\end{eqnarray}
It follows from \eqref{Main_eqlation1} and \eqref{Def_Id_op1} that
\begin{equation}\label{G_lim_2}
\langle \widehat{u}(0)I, v\rangle=\langle I, \widehat{u}(0)v\rangle \overset{\eqref{Main_eqlation1}\eqref{Def_Id_op1}}=\langle  \delta, \widehat{u}(0)f\rangle= \tau_\theta(u)f(0).\nonumber
\end{equation}
That is, we have the convergence
\begin{equation}\label{distributional_convergence_of_heat_flow}
    (4\pi t)^{d/2}G_t\ast u \rightarrow
    \tau_\theta(u)I,\quad t\to\infty
\end{equation}
in the sense of distributions.

From Lemma \ref{HK_Proposition1}\eqref{L_p_L_q_mapping_HK_prop_1}, we also have
$$
    \sup_{t>0}\, (4\pi t)^{\frac{d}{2}}\|G_t\ast u\|_{L^\infty(\mathbb{R}^d_\theta)} \leq \|u\|_{L^1(\mathbb{R}^d_\theta)}.
$$
Since $\mathcal{S}(\mathbb{R}^d_\theta)$ is dense in $L^1(\mathbb{R}^d_\theta)$ (see, \cite[Remark 2.9. p.504]{MSX}), it follows that
$$
    (4\pi t)^{\frac{d}{2}}G_t\ast u\rightarrow \tau_\theta(u)I
$$
in the weak$^*$-topology.
\end{proof}

\begin{proof}[Proof of Theorem \ref{tauberian_type_theorem}]
    Let $t>0$ and let $0\leq w\in L_1(\mathbb{R}^d_\theta)$ have $\tau_\theta(w)=1.$ Let
$$
        v_t := u^{\frac12}((4\pi t)^{\frac{d}{2}}\exp(-t\Delta_\theta)w)u^{\frac12}.
$$
    By Lemma \ref{abelian_type_lemma}, in the weak$^*$-topology we have
$$
        v_t\to u,\quad t\to\infty.
$$
    as $t\to\infty.$ Hence, if $P\in L^1(\mathbb{R}^d_\theta)$ is a projection, we have
$$
        \tau_\theta(PuP) = \lim_{t\to\infty} \tau_\theta(Pv_t) \leq \liminf_{t>0} \tau_\theta(v_t).
$$
    We compute
    \begin{align*}
        \tau_\theta(v_t) &= (4\pi t)^{\frac{d}{2}}\tau_\theta((e^{-t\Delta_\theta}w)u)\\
                  &= (4\pi t)^{\frac{d}{2}}\tau_\theta(w(e^{-t\Delta_\theta}u))\\
                  &\leq \|w\|_{L^1(\mathbb{R}^d_\theta)}\sup_{t>0} (4\pi t)^{\frac{d}{2}}\|e^{-t\Delta_\theta}u\|_{L^\infty(\mathbb{R}^d_\theta)}.
    \end{align*}
    Hence, for any $\tau_\theta$-finite projection $P$ we have
$$
        \tau_\theta(PuP) \leq \sup_{t> 0} (4\pi t)^{\frac{d}{2}}\|G_t\ast u\|_{L^{\infty}(\mathbb{R}^d_\theta)} < \infty.
$$
    By the Fatou property of $L^1(\mathbb{R}^d_\theta)$ (see, \cite[Theorem 3.4.17. p.151]{DPS}), if $u\geq 0,$ then
$$
        \tau_\theta(u)\leq \liminf_{P\uparrow 1} \tau_\theta(PuP),
$$
     with the understanding that if the right hand side is finite, then $u\in L^1(\mathbb{R}^d_\theta).$
     Since $(4\pi t)^{\frac{d}{2}}G_t\ast u$ converges to $\tau_\theta(u)I$ in the weak$^*$ topology, by the uniform boundedness principle we have
$$
        \tau_\theta(u)\leq \sup_{t\geq 0} (4\pi t)^{\frac{d}{2}}\|G_t\ast u\|_{L^\infty(\mathbb{R}^d_\theta)}.
$$
     This completes the proof.
\end{proof}

\section{Nonlinearity estimate}
In this section, we prove a fundamental inequality in general semifinite von Neumann algebras that is essential in the analysis of nonlinear partial differential equations in noncommutative setting, particularly for establishing global existence and local well-posedness.

Let $\mathcal{M}$ be a semifinite von Neumann algebra acting on a Hilbert space $H$, equipped with a faithful normal semifinite trace $\tau$.
Denote by $L^0(\mathcal{M})$ the $\ast$-algebra of all $\tau$-measurable operators affiliated with $\mathcal{M}$. For $1\le p<\infty$, the noncommutative $L^p$-space associated with $(\mathcal{M},\tau)$ is defined by
$$
L^p(\mathcal{M},\tau)
=
\left\{ A \in L^0(\mathcal{M}) : \tau(|A|^p)<\infty \right\},
$$
where $|A|=(A^*A)^{1/2}$.
The space $L^p(\mathcal{M},\tau)$ is a Banach space for $1\le p<\infty,$   equipped with the norm
$$
\|A\|_{L^p(\mathcal{M},\tau)} = \left(\tau(|A|^p)\right)^{1/p}.
$$
For $p=\infty$, one sets $L^\infty(\mathcal{M},\tau)=\mathcal{M}$ endowed with the operator norm. For the general theory of $\tau$-measurable operators
and noncommutative $L^p$-spaces associated with semifinite von Neumann algebras $\mathcal{M}$ we refer the reader to \cite{DPS, LSZ, PXu}.

This section is devoted to the proof of the following important inequality:
\begin{thm}\label{nonlinearity_estimate_theorem}
Let $p\geq 1$ and $1\leq q\leq\infty.$ There exists a constant $c_p>0$ such that for all $0\leq A,B \in L^{p\cdot{q}}(\mathcal{M},\tau),$ we have
\begin{equation}\label{nonlinearity_estimate_thm}
\|A^p-B^p\|_{L^q(\mathcal{M},\tau)} \leq c_p\|A^{p-1}(A-B)+(A-B)B^{p-1}\|_{L^q(\mathcal{M},\tau)},
\end{equation}
where $p\cdot q$ denotes the product of $p$ and $q.$
\end{thm}

From this theorem and the noncommutative Minkowskii \cite{PXu} and  H\"{o}lder inequalities \cite[Theorem 1]{Sukochev}, we immediately obtain the following:
\begin{cor}\label{nonlinearity_estimate_corollary}
Let $p\geq 1$ and $1\leq q\leq\infty.$ There exists a constant $C_p>0$ such that for all $0\leq A,B\in L^{p\cdot{q}}(\mathcal{M},\tau),$ we have
\begin{equation}\label{nonlinearity_estimate}
\|A^p-B^p\|_{L^q(\mathcal{M},\tau)}\leq C_p\|A-B\|_{L^{p\cdot{q}}(\mathcal{M},\tau)}(\|A\|_{L^{p\cdot{q}}(\mathcal{M},\tau)}^{p-1}+\|B\|_{L^{p\cdot{q}}(\mathcal{M},\tau)}^{p-1}),
\end{equation}
\end{cor}
Here we have written $p\cdot q$ for the product $pq$ to avoid confusion with the Lorentz spaces $L^{p,q}(\mathcal{M},\tau)$ which do not appear in this paper.

Similar estimates for $0<p<1$ are well known. See e.g. \cite{BirmanKaradzhovSolomyak, HuangSukochevZanin, Ricard, Sobolev}. For $q=\infty,$ Corollary \ref{nonlinearity_estimate_corollary} follows at once from Peller's Besov space sufficient condition for a function to be operator Lipschitz \cite{Peller}.

For non-positive operators, the best that can be achieved is the following, which is a combination of Corollary \ref{nonlinearity_estimate} and the Lipschitz estimate for the absolute value function on $L^q(\mathcal{M},\tau)$ \cite[Theorem~2.2]{DDdPS}.
\begin{cor}\label{nonlinearity_estimate_nonpositive}
    Let $A,B\in L^{p\cdot q}(\mathcal{M},\tau).$ For all $p\geq 1$ and $1<q<\infty,$ there exists a constant $C_{p,q}$ such that
$$
        \||A|^p-|B|^p\|_{L^q(\mathcal{M},\tau)} \leq C_{p,q}\|A-B\|_{L^{p\cdot q}(\mathcal{M},\tau)}(\|A\|_{L^{p\cdot q}(\mathcal{M},\tau)}^{p-1}+\|B\|_{L^{p\cdot q}(\mathcal{M},\tau)}^{p-1}).
$$
\end{cor}

The proof of Theorem~\ref{nonlinearity_estimate_theorem} relies on the following auxiliary lemmas.
\begin{lem}\label{schwartz_class_lemma}
    Let $p>1.$ The function
$$
        f(t) := \frac{\sinh(\left(\frac{p}{2}-1\right)t)}{\cosh\left(\frac{p-1}{2}t\right)\sinh(\frac{t}{2})},\quad t \in \mathbb{R},
$$
    belongs to the Schwartz class $\mathcal{S}(\mathbb{R}).$
\end{lem}
\begin{proof}
    The identity
$$
        f(t)=1-\frac{\tanh\!\left(\frac{p-1}{2}t\right)}{\tanh(t/2)}
$$
shows that the apparent singularity at the origin is removable; in fact, $f$ extends smoothly to $t=0$ with $f(0)=p-2$. For $|t|\geq1$, the numerator in the original definition satisfies
$$
    \left|\sinh\!\left(\frac{p-2}{2}t\right)\right|\leq C_p e^{\frac{|p-2|}{2}|t|},
$$
whereas the denominator has absolute value bounded below by $c_p e^{\frac p2|t|}$. Hence,
$$
|f(t)|\leq C_p e^{-c_p|t|},\qquad c_p:=\frac{p-|p-2|}{2}>0.
$$
After differentiating the quotient, the same argument gives, for every integer $k\geq0$, constants $C_{p,k}>0$ and $c_p>0$ such that
$$
        |f^{(k)}(t)|\leq C_{p,k}e^{-c_p|t|},\qquad |t|\geq1.
$$
Together with smoothness near the origin, this proves that $f\in\mathcal S(\mathbb R)$.
\end{proof}
Given a bounded Borel function $\phi$ on $\mathbb{R}^2$ and self-adjoint operators $A$ and $B$ affiliated with $\mathcal{M}.$ The double operator integral $T^{A,B}_{\phi}(X)$ is given formally by
\begin{equation}\label{DOI}
T^{A,B}_{\phi}(X)
=
\int\limits_{\mathbb{R}^{2}} \phi(\lambda,\mu) dE_{A}(\lambda) XdE_{B}(\mu),
\quad X\in \mathcal{M}+L^1(\mathcal{M}),
\end{equation}
where $E_{A}$ and $E_{B}$ are the spectral measures for $A$ and $B,$ respectively.  For technical reasons we instead define $T^{A,B}_{\phi}(X)$ only for functions $\phi$ admitting a decomposition of the form
$$
    \phi(\lambda,\mu) = \int\limits_{\Omega} a(\lambda,\omega)b(\mu,\omega) d\nu(\omega),
$$
where $(\Omega,\nu)$ is a probability space and $a,b$ bounded functions. For the theory of double operator integrals developed in this way, see \cite{DDSZII}. We will use the fact that
$$
    \|T^{A,B}_{\phi}(X)\|_{L^{p}(\mathcal{M},\tau)} \leq \int\limits_{\Omega} \sup_{\lambda\in \mathbb{R}} |a(\lambda,\omega)|\sup_{\mu\in \mathbb{R}} |b(\mu,\omega)|\,d|\nu|(\omega) \cdot \|X\|_{L^p(\mathcal{M},\tau)}.
$$
For more details we refer also to \cite{PWS}; see also \cite{BirmanKaradzhovSolomyak} and \cite[Section~3.5]{ST}.

\begin{lem}\label{doi_lemma}Let $p>1$ and $\phi$ be the function defined by
$$
        \phi(\lambda,\mu) := \frac{\lambda^{p}-\mu^p}{(\lambda-\mu)(\lambda^{p-1}+\mu^{p-1})},\quad \lambda,\mu\geq 0,
$$
    with $\phi(\lambda,\mu)=1$ if $\lambda$ or $\mu$ are zero. Then for all positive self-adjoint operators $A$ and $B$ affiliated to $\mathcal{M},$ the double operator integral
    $T^{A,B}_{\phi}$ defined by \eqref{DOI} is bounded on $L^q(\mathcal{M},\tau)$ for every $1\leq q\leq \infty.$
\end{lem}
\begin{proof}
    Since
$$
        (\lambda^{p-1}+\mu^{p-1})(\lambda-\mu) = \lambda^p-\mu^p-\lambda^{p-1}\mu+\mu^{p-1}\lambda, \quad \lambda,\mu\ge 0,
$$
    we have
$$
        \phi = 1+\psi,
$$
    where
$$
        \psi(\lambda,\mu) = \frac{\lambda^{p-1}\mu-\mu^{p-1}\lambda}{(\lambda-\mu)(\lambda^{p-1}+\mu^{p-1})}
$$
    and $\psi(\lambda,\mu)=0$ whenever $\lambda$ or $\mu$ are zero.
    Hence, it suffices to show that $T^{A,B}_{\psi}$ is bounded on $L^q(\mathcal{M}, \tau).$ Writing $\lambda\mu^{-1}=e^t,$ we have
$$
        \psi(\lambda,\mu) = \frac{1}{2}f(t)
$$
    where $f$ is the function from Lemma \ref{schwartz_class_lemma}. Since $f$  belongs Schwartz class  $\mathcal{S}(\mathbb{R}),$ its Fourier transform $\widehat{f}$ is also in $\mathcal{S}(\mathbb{R}).$ Therefore, by the Fourier inversion formula we have
$$
        \psi(\lambda,\mu) = \frac{1}{4\pi}\int_{-\infty}^\infty \widehat{f}(\xi)e^{i\xi t}\,d\xi = \frac{1}{4\pi}\int_{-\infty}^\infty \widehat{f}(\xi)\lambda^{i\xi}\mu^{-i\xi}\,d\xi.
$$
    For $X\in L^q(\mathcal{M},\tau),$ we have
$$
        T_{\psi}^{A,B}(X) = \frac{1}{4\pi}\int_{-\infty}^\infty \widehat{f}(\xi)A^{i\xi}XB^{-i\xi}\,d\xi.
$$
    Observe that this is still true when $A$ and $B$ are not strictly positive, provided we use the convention that $0^{i\xi}=0$ for all $\xi\in \mathbb{R}.$ By the triangle inequality for $L^q(\mathcal{M},\tau),$ it follows that
$$
        \|T_{\psi}^{A,B}(X)\|_{L^q(\mathcal{M},\tau)}\leq \frac{1}{4\pi}\int_{-\infty}^\infty |\widehat{f}(\xi)|\,d\xi \cdot \|X\|_{L^q(\mathcal{M},\tau)}, \,\ 1\leq q\leq \infty.
$$
This completes the proof.
\end{proof}
The proof of Lemma \ref{doi_lemma} also shows that the norm of $T_{\phi}^{A,B}$ is independent of $q,$ which we will use implicitly in the proof of Theorem \ref{nonlinearity_estimate_theorem}, but is ultimately unimportant for our applications.

Having gathered the necessary tools, we are now equipped to prove our main inequality, Theorem~\ref{nonlinearity_estimate_theorem}.
\begin{proof}[Proof of Theorem~\ref{nonlinearity_estimate_theorem}]
    The result is trivial for $p=1,$ hence, assume that $p>1.$
    Let $\phi$ be the function in Lemma \ref{doi_lemma}. We have
$$
        A^p-B^p = T^{A,B}_{\phi}\left(A^{p-1}(A-B)+(A-B)B^{p-1}\right).
$$
    Hence, by Lemma \ref{doi_lemma}, there is a constant $c_p>0$ such that
$$
        \|A^{p}-B^p\|_{L^q(\mathcal{M},\tau)}\leq c_p\|A^{p-1}(A-B)+(A-B)B^{p-1}\|_{L^q(\mathcal{M},\tau)}, \,\ 1\leq q\leq \infty,
$$
    thereby completing the proof.
\end{proof}
\begin{rem}\label{Exp-remark}Note that the inequality in Corollary \ref{nonlinearity_estimate_corollary} is essential in the analysis of nonlinear partial differential equations, particularly for establishing global existence and local well-posedness. By using this inequality we can extend the main results in \cite[Theorem 5.4. p.31]{RST2}, \cite[Theorem 4.2. p. 85]{RST1}, and \cite[Theorem 5.1.2. p. 16]{STT} to the full range $1\leq p<\infty$, which were previously obtained only for positive integer values of the parameter $p.$
\end{rem}

\begin{rem}
    For $p>1,$ the proof of Theorem~\ref{nonlinearity_estimate_theorem} gives the estimate
$$
        c_p\leq \frac{1}{4\pi}\int_{-\infty}^\infty |\widehat{f}(\xi)|\,d\xi
$$
    where $f$ is the function in Lemma \ref{schwartz_class_lemma}.
\end{rem}

\begin{rem}
    A restatement of Theorem~\ref{nonlinearity_estimate_theorem} is that
$$
        A^p-B^p \prec\prec c_p(A^{p-1}(A-B)+(A-B)B^{p-1})
$$
    where $\prec\prec$ is submajorisation in the sense of Hardy-Poly\`a-Littlewood.
\end{rem}

\section{Local existence of the solution}
Let $f:[0,\infty)\to L^{\infty}(\mathbb{R}^d_\theta).$ We say that $f$ is weak$^*$ measurable if the scalar-valued function $t\mapsto \tau(f(t)x)$ is measurable for all $x \in L^1(\mathbb{R}^d_\theta).$ This assumption implies that the norm function $t\mapsto \|f(t)\|_{L^\infty(\mathbb{R}^d_\theta)}$ is measurable, and that there is a weak$^*$ (or Gelfand) integral
$$
    t\mapsto \int_0^t f(s)\,ds
$$
which belongs to $C\big([0,\infty),L^{\infty}(\mathbb{R}^d_\theta)\big),$ and
$$
    \left\|\int_0^t f(s)\,ds\right\|_{L^\infty(\mathbb{R}^d_\theta)} \leq \int_0^t \|f(s)\|_{L^{\infty}(\mathbb{R}^d_\theta)}\,ds.
$$
For details see \cite[p. 53]{DiestelUhl}.

Now, we consider the following local well-posedness result for the nonlinear heat equation on  quantum Euclidean space $L^{\infty}(\mathbb{R}^d_{\theta})$ of the form:
\begin{equation}\label{Main-equation1}
\partial_tu(t)+\Delta_{\theta}u(t)=|u(t)|^p,\quad t>0,\quad p>1,
\end{equation}
with the initial data
\begin{equation}\label{Main-equation2}
u(0)=u_0,
\end{equation}
where $u_0\in L^\infty(\mathbb{R}^d_{\theta})$ and we recall that $\Delta_{\theta}$ is the Laplacian defined by \eqref{laplacian}.
\begin{definition}[Mild solution] Let $1\le q \le \infty$ and $ 0<T<\infty.$ Then, a local mild solution of the problem \eqref{Main-equation1}-\eqref{Main-equation2} is an operator-valued function $u\in C([0, T],L^{q}(\mathbb{R}^d_\theta)\cap L^{\infty}(\mathbb{R}^d_\theta))$ for $q<\infty$ (resp. $u\in C((0, T], L^{\infty}(\mathbb{R}^d_\theta))$ for $q=\infty$)
satisfying the equality
\begin{equation}\label{Main-equation3}
u(t) = e^{-t\Delta_{\theta}}u_0+ \int\limits_0^t e^{-(t-s)\Delta_{\theta}}|u(s)|^p ds,
\end{equation}
for any $t \in [0,T),$ where the integral is taken in the weak$^*$ sense.
\end{definition}
\begin{rem}
    The necessity of using a weak$^*$ integral is due to the fact that $t\mapsto e^{-t\Delta_{\theta}}$ is not a strongly continuous semigroup on $L^\infty(\mathbb{R}^d_\theta).$

    However, since $t\mapsto e^{-t\Delta_\theta}$ is a strongly continuous semigroup on $L^1(\mathbb{R}^d_\theta),$ we do have the weak$^*$ continuity
$$
        \lim_{t\to s} \tau_\theta((e^{-t\Delta_\theta}u)v) = \tau_\theta((e^{-s\Delta_\theta}u)v),\quad s\geq 0,\; v \in L^1(\mathbb{R}^d_\theta).
$$
\end{rem}

Due to the weak$^*$ integral and the failure of strong continuity of the heat semigroup, the following result, while its proof is routine, is not easily referenced. Hence, we include a proof for completeness.
\begin{thm}\label{Local_existence1} (Local existence). Let $\frac{d}{2}< q< \infty$ and $1<p<\infty.$  Let $u_0\in L^{q}(\mathbb{R}^d_\theta)\cap L^{\infty}(\mathbb{R}^d_\theta).$ Then, there exists $T>0$ and
$$
u\in C([0, T],L^{q}(\mathbb{R}^d_\theta))\cap L^{\infty}([0,T],L^{\infty}(\mathbb{R}^d_\theta))
$$
satisfying the problem \eqref{Main-equation1}-\eqref{Main-equation2}.   The solution $u(t)$ can be extended (in a unique way) to a maximal interval $[0,T_{\max}),$ where $0 < T_{\text{max}} \leq \infty.$ Furthermore, if $T_{\text{max}} < \infty,$ then
$$
\liminf\limits_{t\to T^{-}_{\max}}\left(\|u(t)\|_{L^q(\mathbb{R}^d_{\theta})}+\|u(t)\|_{L^{\infty}(\mathbb{R}^d_{\theta})}\right)= \infty.
$$
\end{thm}
\begin{proof}
Let
$$
    X_T := C([0,T],L^q(\mathbb{R}^d_\theta))\cap L^\infty([0,T],L^\infty(\mathbb{R}^d_\theta)))
$$
be equipped with the norm
$$
    \|u\|_{X_T} := \sup_{0\leq t,s\leq T} (\|u(t)\|_{L^q(\mathbb{R}^d_\theta)}+\|u(s)\|_{L^\infty(\mathbb{R}^d_\theta)}).
$$
For $0 < t \leq T,$ we introduce the map $\mathcal{K}$
$$
(\mathcal{K} u)(t) = e^{-t\Delta_{\theta}}u_0 + \int\limits_0^t e^{-(t-s)\Delta_{\theta}}(|u(s)|^p)ds, \quad u \in  X_T.
$$
Observe first that since $\exp(-t\Delta_{\theta})$ is contractive in $L^\infty(\mathbb{R}^d_\theta)$ (Lemma~\ref{HK_Proposition1}\eqref{L_p_L_p_mapping_HK_prop_1}) we have
\begin{equation}\label{L_infty_bound_on_K}
    \|\mathcal{K}u(t)\|_{L^\infty(\mathbb{R}^d_\theta)} \leq \|u_0\|_{L^\infty(\mathbb{R}^d_\theta)}+t\sup_{0\leq s\leq T} \|u(s)\|_{L^\infty(\mathbb{R}^d_\theta)}^p
\end{equation}
So that $\mathcal{K}_1u \in L^\infty([0,T],L^\infty(\mathbb{R}^d_\theta)).$ Since $\exp(-t\Delta_\theta)$ is also contractive in $L^q(\mathbb{R}^d_\theta),$ we have
\begin{align}
    \|\mathcal{K}u(t)\|_{L^\infty(\mathbb{R}^d_\theta)} &\leq \|u_0\|_{L^q(\mathbb{R}^d_\theta)} + t\sup_{0\leq s\leq T} \||u(s)|^p\|_{L^q(\mathbb{R}^d_\theta)}\nonumber\\
    &\leq \|u_0\|_{L^q(\mathbb{R}^d_\theta)}+t\sup_{0\leq s\leq T}\|u(s)\|_{L^\infty(\mathbb{R}^d_\theta)}^{p-1}\|u(s)\|_{L^q(\mathbb{R}^d_\theta)}.\label{L_q_bound_on_K}
\end{align}
The combination of \eqref{L_infty_bound_on_K} and \eqref{L_q_bound_on_K} gives
\begin{equation}\label{X_T_bound_on_K}
    \|\mathcal{K}u(t)\|_{X_T} \leq \|u_0\|_{L^q(\mathbb{R}^d_\theta)}+\|u_0\|_{L^\infty(\mathbb{R}^d_\theta)}+T\|u\|_{X_T}^p.
\end{equation}
Similarly, since $\exp(-t\Delta_\theta)$ is strongly continuous on $L^q(\mathbb{R}^d_\theta)$ for $q<\infty,$ we have that
$$
    \mathcal{K}u \in C([0,T],L^q(\mathbb{R}^d_\theta)).
$$
Therefore $\mathcal{K}:X_T\to X_T.$
Let $\delta := \|u_0\|_{L^q(\mathbb{R}^d_\theta)}+\|u_0\|_{L^\infty(\mathbb{R}^d_\theta)},$ and denote
$$
\mathcal{X}_\delta = \{u \in X_T\;:\; \|u\|_{X_T} \leq 2\delta\},
$$
The set $\mathcal{X}_\delta$ is a closed subspace of $X_T,$ and is hence a complete metric space.

From \eqref{X_T_bound_on_K}, if $u\in \mathcal{X}_1,$ we have
$$
    \|\mathcal{K}u\|_{X_T} \leq \delta+T(2\delta)^p.
$$
Choosing $T$ sufficiently small, depending on $\delta,$ we can ensure that
$$
    \mathcal{K}:\mathcal{X}_\delta\to \mathcal{X}_\delta.
$$
Let us now show that $\mathcal{K}$ is a contraction map on $\mathcal{X}_\delta.$ Let $u,v\in\mathcal{X}_\delta.$ Once again applying the triangle inequality together with the contractivity of $\exp(-t\Delta_\theta)$ on $L^q(\mathbb{R}^d_\theta),$ for all $0<t<T,$ we have
\begin{eqnarray}\label{Eq_1}
    \|(\mathcal{K}u)(t) - (\mathcal{K}v)(t)\|_{L^{q}(\mathbb{R}^d_{\theta})}
     &\le&  \int\limits_{0}^{t} \|e^{-(t-s)\Delta_{\theta}} (|u(s)|^p-|v(s)|^p)\|_{L^{q}(\mathbb{R}^d_{\theta})}ds \nonumber\\
     &\le&  \int\limits_{0}^{t} \||u(s)|^p-|v(s)|^p\|_{L^{q}(\mathbb{R}^d_{\theta})}ds.
\end{eqnarray}
Applying Corollary \ref{nonlinearity_estimate_nonpositive} for all $0<s<T$ we have
\begin{eqnarray}\label{Eq_2}
\||u(s)|^p-|v(s)|^p\|_{L^{q}(\mathbb{R}^d_{\theta})}&\overset{\eqref{nonlinearity_estimate_thm}}\le& C_{p,q}\|u(s)-v(s)\|_{L^{q}(\mathbb{R}^d_{\theta})}\left(\|u(s)\|^{p-1}_{L^{\infty}(\mathbb{R}^d_{\theta})}+\|v(s)\|^{p-1}_{L^{\infty}(\mathbb{R}^d_{\theta})}\right)\nonumber\\
&\le&2^p\delta^{p-1}C_{p,q}\|u(s)-v(s)\|_{L^{q}(\mathbb{R}^d_{\theta})}.\nonumber
\end{eqnarray}
Hence, combining \eqref{Eq_1} and \eqref{Eq_2} we obtain
$$
\|(\mathcal{K}u)(t) - (\mathcal{K}v)(t)\|_{L^{q}(\mathbb{R}^d_{\theta})} \le  2^{p} C_{p,q} T\delta^{p-1} \|u - v\|_{C([0, T], L^q(\mathbb{R}^d_{\theta}))}.
$$
Now we use the boundedness of $\exp(-t\Delta_\theta)$ from $L^{q}(\mathbb{R}^d_\theta)$ into $L^\infty(\mathbb{R}^d_\theta)$ from Lemma~\ref{HK_Proposition1}.\eqref{L_p_L_q_mapping_HK_prop_1}. We have
$$
    \|(\mathcal{K}u)(t)-(\mathcal{K}v)(t)\|_{L^\infty(\mathbb{R}^d_\theta)} \leq \int_0^t (t-s)^{-\frac{d}{2q}} \||u(s)|^p-|v(s)|^p\|_{L^{q}(\mathbb{R}^d_\theta)}\,ds.
$$
Since $q>\frac{d}{2},$ this estimate and Corollary \ref{nonlinearity_estimate_nonpositive}gives
$$
    \|(\mathcal{K}u)(t)-(\mathcal{K}v)(t)\|_{L^\infty(\mathbb{R}^d_\theta)} \leq T^{1-\frac{d}{2q}}2^p\delta^{p-1}\sup_{0\leq s\leq T}\|u(s)-v(s)\|_{L^q(\mathbb{R}^d_\theta)}.
$$
Combining these estimates we have
$$
    \|\mathcal{K}u-\mathcal{K}v\|_{X_T} \leq 2^p(T+T^{1-\frac{d}{2q}})\delta^{p-1}\|u-v\|_{X_T}.
$$
Therefore, $\mathcal{K}$ is a contraction mapping on $\mathcal{X}_\delta$ provided that $T$ is again made sufficiently small (depending on $\delta$).

Finally, the Banach Fixed Point Theorem implies that $\mathcal{K}$ has a unique
fixed point $u\in  \mathcal{X}_\delta.$ That is, there exists a unique $u\in \mathcal{X}_\delta$ such that for all $0<t<T$ we have
$$
u(t)=(\mathcal{K} u)(t) = e^{-t\Delta_{\theta}}u_0 + \int\limits_0^t e^{-(t-s)\Delta_{\theta}}(|u(s)|^p)ds.
$$
Thus, it follows that problem \eqref{Main-equation1}-\eqref{Main-equation2} has a unique mild solution $u\in \mathcal{X}_1 \subset C((0, T], L^\infty(\mathbb{R}^d_{\theta})).$   This completes the proof of the existence of a mild solution.

The preceding argument shows that if $\varepsilon<C_{p,q}^{-1}2^{-p-1(}\|u(T)\|_{L^\infty(\mathbb{R}^d_\theta)}+\|u(T)\|_{L^q(\mathbb{R}^d_\theta)})^{1-p},$ then a mild solution in the interval $[0, T]$ can be extended to $[0,T+\varepsilon).$ Thus if $T<\infty$ and
$$
\sup\limits_{0<t<T}\;(\|u(t)\|_{L^{\infty}(\mathbb{R}^d_\theta)}+\|u(t)\|_{L^{q}(\mathbb{R}^d_\theta)})<\infty,
$$
then $T$ is not maximal. We deduce that if $T_{\max}<\infty,$ then
$$
\liminf\limits_{t\to T^{-}_{\max}}\; (\|u(t)\|_{L^\infty(\mathbb{R}^d_\theta)}+\|u(t)\|_{L^{q}(\mathbb{R}^d_\theta)}) = +\infty.
$$
This completes the proof.
\end{proof}

\begin{rem}   Assume that $u_0 \ge 0$. Then one can obtain a nonnegative solution on some time interval $[0,T]$ by performing the fixed point construction inside the positive cone
$$
\mathcal{X}_\delta^+:=\{u\in \mathcal{X}_\delta:\ u\ge 0\},
$$
and exploiting the positivity of the operator $e^{-t\Delta_{\theta}}$ (see, Corollary \ref{heat-cp-map}).
By uniqueness of solutions, we conclude that if $u_0\geq 0,$ then
$$
u(t)\ge 0 \qquad \text{for all } t\in (0,T_{\max}).
$$
\end{rem}
The following theorem gives a local existence of positive solutions.
\begin{thm}\label{Local_existence} (Local existence). Let $1\le  q<  \infty$ and $1< p< \infty.$  Let $0\le u_0\in L^{q}(\mathbb{R}^d_\theta)\cap L^{\infty}(\mathbb{R}^d_\theta).$ Then, there exists $T>0$ and a unique mild solution
$$
    u\in C([0, T],L^{q}(\mathbb{R}^d_\theta))\cap L^\infty([0,T],L^\infty(\mathbb{R}^d_\theta))
$$
for the problem \eqref{Main-equation1}-\eqref{Main-equation2}.   The solution $u(t)$ can be extended (in a unique way) to a maximal interval $(0,T_{\max}),$ where $0 < T_{\text{max}} \leq \infty.$ Furthermore, if $T_{\text{max}} < \infty,$ then
$$
\liminf\limits_{t\to T^{-}_{\max}}\left(\|u(t)\|_{L^q(\mathbb{R}^d_{\theta})}+\|u(t)\|_{L^{\infty}(\mathbb{R}^d_{\theta})}\right)= \infty.
$$
\end{thm}
 \begin{proof}
 Applying Theorem~\ref{nonlinearity_estimate_theorem} and repeating the proof strategy of Theorem~\ref{Local_existence1} but replacing the use of Corollary \ref{nonlinearity_estimate_nonpositive} by Corollary \ref{nonlinearity_estimate_corollary}, we deduce Theorem~\ref{Local_existence}.
\end{proof}

\begin{rem}In Theorem~\ref{Local_existence}, the cases $1\leq q\leq \frac{d}{2}$ can be included since we use Corollary \ref{nonlinearity_estimate} in place of Corollary \ref{nonlinearity_estimate_nonpositive}.
 \end{rem}

\section{Auxiliary results}\label{Se:6}
 In this section, we establish two auxiliary results, namely Lemma~\ref{Lemma5_5} and
Corollary~\ref{Corollary5_5}. These statements may be regarded as noncommutative analogues of \cite[Lemma~15.6, p.~94 and Corollary~15.8, p.~96]{QS}, where the corresponding results are proved in the commutative setting. The results presented
here serve as essential technical tools and will play a crucial role in the proof of Theorem~\ref{Mainthm6_2}, which is carried out in Section~\ref{Se:7}.

 \begin{lem}\label{Lemma5_5}  Let $1<p\leq2$ and $T>0$. Let $0\leq u_0\in L^\infty(\mathbb{R}^d_\theta)$. If there exists a positive operator-valued function $u\in C((0,T],L^\infty(\mathbb{R}^d_\theta))$ satisfying
\begin{equation}\label{Lem_eq_1}
u(t) \geq e^{-t\Delta_{\theta}}u_0+ \int\limits_0^te^{-(t-s)\Delta_{\theta}} u^p(s)ds,\quad 0\leq t \leq T,
\end{equation}
then we have
\begin{equation}\label{Lem_eq_2}
t^{\frac{1}{p-1}}\|e^{-t\Delta_{\theta}}u_0\|_{L^\infty(\mathbb{R}^d_{\theta})} \le  (p-1)^{-\frac{1}{p-1}} \quad \text{ for all} \quad  0\leq t \leq T.
\end{equation}

\end{lem}
\begin{proof}  For fixed  \( 0\neq v \in L^1(\mathbb{R}^d_{\theta})\), we introduce a  linear functional $\Phi_{\theta}: L^{\infty}(\mathbb{R}^d_{\theta})\to \mathbb{C}$ defined by
\begin{equation}\label{linear_functional}
\Phi_{\theta}(u):=\tau_{\theta} \left(u\cdot\frac{|v|}{\|v\|_{L^1(\mathbb{R}^d_{\theta})} }\right),
\qquad
u \in  L^{\infty}(\mathbb{R}^d_{\theta}),
\end{equation}
which is a unital completely positive linear map. Suppose $0\leq t \le s_1\le T.$  Then, we set
\begin{equation}\label{Lem_eq_4}
\rho (t):=\Phi_{\theta}\left(e^{-s_1\Delta_{\theta}} u_0\right) + \int\limits_0^t   \Phi_{\theta}\left(\left(e^{-(s_1 - s)\Delta_{\theta}} u(s) \right)^p\right)ds,  \quad t \in [0, s_1].
\end{equation}
Differentiating with respect to $t,$ we obtain
\begin{equation}\label{Lem_eq_5}
\partial_{t}\rho (t)=\Phi_{\theta}\left(\left(e^{-(s_1 - t)\Delta_{\theta}} u(t) \right)^p\right),\quad t \in [0,s_1].
\end{equation}
Then, it follows from \eqref{inequality_1} and \eqref{Lem_eq_1}  that
\begin{eqnarray} \label{Lem_eq_6}
  e^{-s_1\Delta_{\theta}} u_0 + \int\limits_0^t \left(e^{-(s_1 - s)\Delta_{\theta}} u(s)\right)^p  ds&\overset{\eqref{inequality_1}}{\leq}&  e^{-s_1\Delta_{\theta}} u_0 + \int\limits_0^t e^{-(s_1 - s)\Delta_{\theta}} u^p(s) ds \nonumber\\
&=&  e^{-(s_1 - t)\Delta_{\theta}}\left(e^{-t\Delta_{\theta}} u_0 + \int\limits_0^t e^{-(t- s)\Delta_{\theta}} u^p(s) ds\right) \\
&\overset{\eqref{Lem_eq_1}}{\leq}&  e^{-(s_1 - t)\Delta_{\theta}} u(t),\nonumber
\end{eqnarray}
for $ t \in [0,s_1].$ Since the map $\Phi_{\theta}$ is order preserving, and applying \eqref{inequality_1}, \eqref{Lem_eq_5} and \eqref{Lem_eq_6}  we obtain
$$
 \rho^p(t)\overset{\eqref{Lem_eq_6}}\leq \Phi^p_{\theta}\left( e^{-(s_1 - t)\Delta_{\theta}} u(t) \right)\overset{\eqref{inequality_1}}\le \Phi_{\theta}\left((e^{-(s_1 - t)\Delta_{\theta}} u(t))^p\right)\overset{\eqref{Lem_eq_5}}=\partial_t\rho(t),\quad t \in [0,s_1].
$$
 Thus,
\begin{equation}\label{Lem_eq_7}
 \partial_t (\rho^{1-p}(t))=-(p-1)\partial_t\rho(t)\rho^{-p}(t)\le-(p-1).
\end{equation}
Consequently, we have
$$
\rho^{1-p}(0)=\rho^{1-p}(s_1)-
\int\limits_0^{s_1} \partial_s (\rho^{1-p}(s)) ds\overset{\eqref{Lem_eq_7} }\ge s_1(p-1).
$$
For $0\neq v\in L^{1}(\mathbb{R}^{d}_{\theta}),$ it follows from  \eqref{linear_functional} and the above calculation that
$$
 \left(\tau_{\theta}((e^{-{s_1}\Delta_{\theta}} u_0)\frac{|v|}{\|v\|_{L^{1}(\mathbb{R}^{d}_{\theta})}})\right)^{p-1}\overset{\eqref{linear_functional}}{=}  \Phi^{p-1}_{\theta}\left( e^{-{s_1}\Delta_{\theta}} u_0\right)=\rho^{p-1}(0)\le s^{-1}_1(p-1)^{-1}.
$$
Therefore, by the identity $L^\infty(\mathbb{R}^d_\theta) = (L^1(\mathbb{R}^d_\theta))^*$ (see \cite[Theorem 3.4.24 . p.165]{DPS} for the general case) we have
\begin{eqnarray*}
\|e^{-{s_1}\Delta_{\theta}} u_0\|_{L^{\infty}(\mathbb{R}^d_{\theta})}&=&\sup\limits_{\|\tilde{v}\|_{L^{1}(\mathbb{R}^{d}_{\theta})}\le1 }|\tau_{\theta}((e^{-{s_1}\Delta_{\theta}} u_0)\tilde{v})|\\
&\le&\sup\limits_{\|\tilde{v}\|_{L^{1}(\mathbb{R}^{d}_{\theta})}\le1 }\tau_{\theta}((e^{-{s_1}\Delta_{\theta}} u_0)|\tilde{v}|)\\
&\le& s^{-\frac{1}{p-1}}_1(p-1)^{-\frac{1}{p-1}}.
\end{eqnarray*}
Since $s_1\le T$ is arbitrary, we obtain that \eqref{Lem_eq_2} holds for every $t\in(0,T]$. This completes the proof.
\end{proof}
\begin{cor}\label{Corollary5_5} Let $1<p\leq2$ and $0<T<\infty$. Let $0\neq u_0\in L^\infty(\mathbb{R}^d_\theta)$ be positive. If there exists a nonnegative operator-valued function $u\in C((0,T],L^\infty(\mathbb{R}^d_\theta))$ satisfying
$$
u(t) = e^{-t\Delta_{\theta}}u_0+ \int\limits_0^te^{-(t-s)\Delta_{\theta}} u^p(s)ds,\quad 0\leq t \leq T,
$$
then for all  $t \in (0, T - s_1]$  and  $s_1 \in (0, T),$ we have
\begin{equation}\label{eq:5.24}
t^\frac{1}{p-1} \|e^{-t\Delta_{\theta}} u(s_1) \|_{L^\infty(\mathbb{R}^d_{\theta})} \leq  (p-1)^{-\frac{1}{p-1}}.
\end{equation}

\end{cor}
\begin{proof} Let $u \in C((0, T], L^\infty(\mathbb{R}^d_{\theta})).$ For   $s_1>0 ,$ we denote  $v(t):= u(t+s_1).$  Thus,
\begin{eqnarray*}
v(t) &=& e^{-(t+s_1)\Delta_{\theta}}u_0 + \int\limits_0^{t+s_1} e^{-(t+s_1-s)\Delta_{\theta}}u^p(s)ds\\
&=& e^{-t\Delta_{\theta}} e^{-s_1\Delta_{\theta}} u_0 + \int_0^{s_1} e^{-t\Delta_{\theta}} e^{-(s_1-s)\Delta_{\theta}}u^p(s)ds + \int\limits_{s_1}^{t+s_1} e^{-(t+s_1-s)\Delta_{\theta}}u^p(s) ds\\
&=& e^{-t\Delta_{\theta}} \left( e^{-s_1\Delta_{\theta}} u_0 + \int\limits_0^{s_1} e^{-(s_1-s)\Delta_{\theta}}u^p(s)ds \right) + \int\limits_0^{t} e^{-(t-s)\Delta_{\theta}} v^p(s) ds \\
&=& e^{-t\Delta_{\theta}} u(s_1) + \int\limits_0^{t} e^{-(t-s)\Delta_{\theta}} v^p(s)ds.
\end{eqnarray*}
Thus, Lemma \ref{Lemma5_5}, with $u_0$ replaced by $u(s_1)$ and $T$ replaced by $T-s_1$, applies for every $s_1\in(0,T)$. This completes the proof.
\end{proof}

\section{Fujita-type results}\label{Se:7}

In this section, we establish a noncommutative analogue of Fujita's theorem \cite{Fujita1966} (see also \cite[Theorem 18.1, p.~113]{QS}). We show that no non-trivial nonnegative global solution exists when $1<p\leq p_F$, whereas, when $p>p_F$, sufficiently small nonnegative initial data give rise to a global mild solution. Thus the critical exponent is
$$
p_F:=1+\frac{2}{d}.
$$
As usual, a positive solution of \eqref{Main-equation1}--\eqref{Main-equation2} is called global if $\|u(t)\|_{L^\infty(\mathbb{R}^d_\theta)}<\infty$ for every $t>0$. If its maximal existence time $T_{\max}$ is finite and $\|u(t)\|_{L^\infty(\mathbb{R}^d_\theta)}\to\infty$ as $t\uparrow T_{\max}$, then the solution is said to blow up in finite time.
First, we show a sufficient condition on the initial data which guarantees the existence of global solutions of the problem \eqref{Main-equation1}-\eqref{Main-equation2} on $\mathbb{R}^d_{\theta}.$
\begin{definition} For $p\leq q\leq\infty,$  a \emph{global mild solution} of the Cauchy problem \eqref{Main-equation1}-\eqref{Main-equation2} in  $L^q(\mathbb{R}^d_\theta)$ is a function   $u : [0, \infty) \to L^q(\mathbb{R}^d_\theta)$  satisfying
\begin{equation}\label{Int_equation}
u(t) = e^{-t\Delta_{\theta}}u_0+ \int\limits_0^t e^{-(t-s)\Delta_{\theta}}  u^p(s) ds,
\end{equation}
where the integral is taken in the weak$^*$ sense in the event that $q=\infty,$ or the $L^{\frac{q}{p}}(\mathbb{R}^d_\theta)$-valued Bochner sense when $q<\infty.$
\end{definition}

We now present the Fujita-type results (see \cite{Fujita1966}, \cite[Theorem 18.1. p.113]{QS}) on the noncommutative Euclidean space $L^{\infty}(\mathbb{R}^d_{\theta})$, characterizing the threshold between global existence and finite-time blow-up of mild solution of \eqref{Main-equation1}-\eqref{Main-equation2}.

\begin{thm}\label{Mainthm6_2} Let $1<p<\infty$ and $p_F := 1 + \frac{2}{d}.$ We have
\begin{enumerate}
    \item[(i)] If $p<p_F ,$  then the differential inequality
\begin{equation}\label{eq:5.14}
\partial_tu(t)+\Delta_{\theta}u(t)\geq  u^p(t),
\end{equation}
does not admit any non-trivial distributional solution $u \geq 0$ in $C\left((0, \infty),   L^\infty(\mathbb{R}^d_{\theta})\right).$
    \item[(ii)] If $p=p_F ,$ then the equation
\begin{equation}\label{eq:5.15}
\partial_tu(t)+\Delta_{\theta}u(t)= u^p(t),
\end{equation}
does not admit any non-trivial distributional solution $u \geq 0$ in $C\left((0, \infty),  L^\infty(\mathbb{R}^d_{\theta})\right).$
    \item[(iii)]  If $p>p_F$  and $ 0\le u_{0}\in L^{\infty}(\mathbb{R}^d_{\theta}) \cap L^{r}(\mathbb{R}^d_{\theta})$ and
$$
        \|u_0\|_{L^r(\mathbb{R}^d_\theta)}+\|u_0\|_{L^\infty(\mathbb{R}^d_\theta)}
$$
    is sufficiently small, where  $r:= \frac{d(p-1)}{2},$ then \eqref{Main-equation1}-\eqref{Main-equation2} has a global mild solution.
\end{enumerate}
\end{thm}

 \begin{rem}\label{Remark5_3} In Part (i) of Theorem \ref{Mainthm6_2}, by a distributional solution we mean an element $u \in C\left((0,\infty),L^\infty(\mathbb{R}^d_\theta)\right)$ such that for all $0\leq v \in \mathcal{S}(\mathbb{R}^d_\theta)$ and $t>0$ we have
$$
    (\partial_t u(t),v)+(\Delta_\theta u(t),v)\geq (u^p(t),v).
$$
 Similarly, in part (ii) we mean the same with the inequality replaced by an equality.
\end{rem}
The key to proving Parts (i) and (ii) of Theorem \ref{Mainthm6_2} is the following result. Our proof mirrors the commutative argument from \cite[Theorem 18.3. p.115]{QS} and \cite{Weissler1981}.
\begin{lem}\label{lemma_divergence_of_integral}
    Let $0\leq u_0\in L_1(\mathbb{R}^d_\theta)$ be non-zero. Let $1<p<\infty.$ We have
$$
        \int\limits_1^{\infty} \|e^{-s\Delta_\theta}u_0\|_{L^p(\mathbb{R}^d_\theta)}^p\,ds = +\infty
$$
    if and only if $p\leq p_F := 1+\frac{2}{d}.$
\end{lem}
\begin{proof}
    By Lemma \ref{HK_Proposition1}\eqref{L_p_L_q_mapping_HK_prop_1}, we have
$$
        (4\pi s)^{\frac{d}{2}}\|e^{-2s\Delta_\theta}u_0\|_{L^\infty(\mathbb{R}^d_\theta)}^p \leq \|e^{-s\Delta_\theta}u_0\|_{L^p(\mathbb{R}^d_\theta)}^p \leq (4\pi s)^{-\frac{d}{2}(p-1)}\|u_0\|_{L^1(\mathbb{R}^d_\theta)}^p.
$$
    Equivalently,
$$
        (4\pi s)^{-\frac{d}{2}(p-1)}\cdot ((4\pi s)^{\frac{d}{2}}\|e^{-2s\Delta_\theta}u_0\|_{L^\infty(\mathbb{R}^d_\theta)})^p \leq \|e^{-s\Delta_\theta}u_0\|_{L^p(\mathbb{R}^d_\theta)}^p \leq (4\pi s)^{-\frac{d}{2}(p-1)}\|u_0\|_{L^1(\mathbb{R}^d_\theta)}^p.
$$
    Note that
$$
        \liminf\limits_{s\to \infty} (4\pi s)^{\frac{d}{2}}\|e^{-s\Delta_\theta}u_0\|_{L^{\infty}(\mathbb{R}^d_\theta)} > 0.
$$
    Indeed, otherwise there exists an increasing sequence $\{s_k\}_{k=0}^\infty$ such that
$$
        \lim\limits_{k\to\infty} (4\pi s_k)^{\frac{d}{2}}\|e^{-s_k\Delta}u_0\|_{L^\infty(\mathbb{R}^d_\theta)} = 0.
$$
    That is, $(4\pi s_k)^{\frac{d}{2}}e^{-s_k\Delta_\theta}u_0\to 0$ in the $L^{\infty}(\mathbb{R}^d_\theta)$ norm. Since the weak$^*$ limit is $\tau_\theta(u_0)I\neq 0,$ this is impossible. Hence there exists $\varepsilon>0$ such that for all sufficiently large $s$ we have
$$
        (4\pi s)^{-\frac{d}{2}(p-1)}\varepsilon \leq \|e^{-s\Delta_\theta}u_0\|_{L^p(\mathbb{R}^d_\theta)}^p \leq (4\pi s)^{-\frac{d}{2}(p-1)}\|u_0\|_{L^1(\mathbb{R}^d_\theta)}^p.
$$
    It follow that $p\leq p_F$ is equivalent to
    $$
    \int\limits_1^\infty \|e^{-s\Delta_\theta}u_0\|_{L^p(\mathbb{R}^d_\theta)}^p \,ds = +\infty.
    $$
\end{proof}

\begin{thm}\label{Alternative_Mainthm5_4}
    Let $0\leq u_0\in L_1(\mathbb{R}^d_\theta)$, $u_0\neq0$, and $1<p\leq p_F:=1+\frac2d$. Then there is no nonnegative $u\in C((0,\infty),L^\infty(\mathbb{R}^d_\theta))$ such that
$$
        u(t)\geq e^{-t\Delta_\theta}u_0+\int\limits_0^t e^{-(t-s)\Delta_\theta}u(s)^p\,ds,\quad t>0.
$$
\end{thm}
\begin{proof}
    Suppose that $u\in C((0,\infty),L^{\infty}(\mathbb{R}^d_\theta))$ exists.
 Applying Lemma \ref{Lemma5_5} to the time-shifted supersolution $v(t):=u(t+s)$, and then letting the terminal time be arbitrary, we obtain
$$
\sup_{t>0} t^{\frac{1}{p-1}}\|e^{-t\Delta_\theta}u(s)\|_{L^\infty(\mathbb{R}^d_\theta)}\leq(p-1)^{-\frac1{p-1}},\qquad s>0.
$$
    Since $p\leq p_F,$ it follows that
$$
\limsup_{t\to\infty}t^{\frac d2}\|e^{-t\Delta_\theta}u(s)\|_{L^\infty(\mathbb{R}^d_\theta)}\leq(p-1)^{-\frac1{p-1}}.
$$
For $0<t\leq1$, contractivity gives $t^{d/2}\|e^{-t\Delta_\theta}u(s)\|_\infty\leq\|u(s)\|_\infty$. Hence Theorem \ref{tauberian_type_theorem} applies and gives $u(s)\in L^1(\mathbb{R}^d_\theta)$ for every $s>0$. Moreover, using the last inequality in the limit $t\to\infty$ yields the uniform estimate
$$
   \sup_{s>0}\tau_\theta(u(s))\leq(4\pi)^{d/2}(p-1)^{-\frac1{p-1}}<\infty.
$$

    By definition
$$
        u(s)\geq e^{-s\Delta_\theta}u_0,\quad s>0
$$
    and it follows that
$$
        \|u(s)\|_{L^p(\mathbb{R}^d_\theta)}^p \geq \|e^{-s\Delta_\theta}u_0\|_{L^p(\mathbb{R}^d_\theta)}^p\quad s > 0.
$$
    Equivalently
$$
        \tau_\theta(u(s)^p) \geq \|e^{-s\Delta_\theta}u_0\|_{L^p(\mathbb{R}^d_\theta)}^p.
$$
By the corresponding time-shifted inequality,
$$
u(t+1)\geq e^{-t\Delta_{\theta}}u(1)+\int\limits_1^{t+1} e^{-(t+1-s)\Delta_\theta}u(s)^pds,
$$
    we have
    \begin{align*}
        \tau_\theta(u(t+1)) &\geq \int\limits_1^{t+1}\tau_\theta(e^{-(t+1-s)\Delta_\theta}(u(s)^p))\,ds\\
        &= \int\limits_1^{t+1} \tau_\theta(u(s)^p)\,ds\\
        &\geq \int\limits_1^{t+1} \|e^{-s\Delta_\theta}u_0\|_{L^p(\mathbb{R}^d_\theta)}^p\,ds
    \end{align*}
    which diverges as $t\to\infty$ by Lemma \ref{lemma_divergence_of_integral}. But this is a contradiction.
\end{proof}

 We now proceed to prove Theorem \ref{Mainthm6_2}.
\begin{proof}[Proof of Theorem \ref{Mainthm6_2}]
We first prove (i) and (ii). Suppose that $u\geq0$ is a non-trivial global distributional solution of \eqref{eq:5.14}, or of \eqref{eq:5.15} in the critical case. Choose $s_0>0$ such that $u(s_0)\neq0$ and set
$$
u_0:=u(s_0),\qquad v(t):=u(t+s_0).
$$
The standard variation-of-constants argument, obtained by testing the distributional inequality against the backward heat evolution of a nonnegative Schwartz operator, gives
$$
v(t)\geq e^{-t\Delta_\theta}u_0+\int_0^t e^{-(t-s)\Delta_\theta}v(s)^pds,\quad t>0.
$$
In the equality case, equality holds in this formula. Lemma \ref{Lemma5_5}, applied on every finite interval, yields
$$
\sup_{t>0}t^{\frac1{p-1}}\|e^{-t\Delta_\theta}u_0\|_\infty<\infty.
$$
Since $p\leq p_F$, we have $d/2\leq1/(p-1)$; contractivity for $0<t\leq1$ and the preceding estimate for $t\geq1$ therefore imply
$$
\sup_{t>0}t^{d/2}\|e^{-t\Delta_\theta}u_0\|_\infty<\infty.
$$
Theorem \ref{tauberian_type_theorem} now gives $u_0\in L^1(\mathbb R^d_\theta)$. This contradicts Theorem \ref{Alternative_Mainthm5_4}, proving (i) and (ii).

We now prove (iii). Assume $p>p_F$ and set
\begin{equation}\label{4.10}
r:=\frac{d(p-1)}{2}>1.
\end{equation}
Since $r>1$, we have $pr>\max\{p,r\}$. We may therefore choose $q$ so that
\begin{equation}\label{4.12}
\max\{p,r\}<q<pr.
\end{equation}
Equivalently,
\begin{equation}\label{4.14}
\frac1{p(p-1)}<\frac{d}{2q}<\frac1{p-1}.
\end{equation}
Let $\beta:=\frac d2(\frac1r-\frac1q)$. Then
\begin{equation}\label{4.15}
0<p\beta<1,\qquad \frac{d(p-1)}{2q}+(p-1)\beta=1.
\end{equation}
For $u_0 \in L^{r}(\mathbb{R}^d_{\theta})$, Lemma \ref{HK_Proposition1} \eqref{L_p_L_q_mapping_HK_prop_1} implies
$$
\sup\limits_{t>0} t^{\beta} \|e^{-t\Delta_{\theta}}u_0\|_{L^q(\mathbb{R}^d_{\theta})} \le  C\|u_0\|_{L^{r}(\mathbb{R}^d_{\theta})} =: \alpha < \infty.
$$
We denote
\begin{equation}
\Omega_{\beta}:= \left\{ u \in C\left((0,\infty), L^{q}(\mathbb{R}^d_{\theta})\right):
\sup_{t>0} t^{\beta} \|u(t)\|_{L^{q}(\mathbb{R}^d_{\theta})} \le M\text{ and } u(t)\geq 0\right\},
\end{equation}
where $M>0$ will be chosen sufficiently small. For any $u,v \in \Omega_{\beta}$, we introduce the metric
\begin{equation}\label{eq:5.7}
d_{\Omega_{\beta}}(u,v)
:= \sup\limits_{t>0} t^{\beta}\,\|u(t)-v(t)\|_{L^{q}(\mathbb{R}^d_{\theta})}.
\end{equation}
It is easy to check that $\bigl(\Omega_{\beta}, d_{\Omega_{\beta}}\bigr)$ is a nonempty complete metric space.

Moreover, for   $u \in \Omega_{\beta}$, we define the map $\mathcal{A}$ by
$$
\mathcal{A}(u)(t)
= e^{-t\Delta_{\theta}}u_0
  + \int\limits_{0}^{t} e^{-(t-s)\Delta_{\theta}}\,u^{p}(s)\,ds.
$$

Since  $q > p>1,$ we have $0<\frac{p}{q} - \frac{1}{q} < 1.$ Hence, by using the triangle inequality and  Lemma \ref{HK_Proposition1} \eqref{L_p_L_q_mapping_HK_prop_1}, we obtain
\begin{eqnarray}\label{Eq:5.32}
t^{\beta} \| \mathcal{A}(u)(t)\|_{L^q(\mathbb{R}^d_{\theta})}
&\le& \alpha+ t^{\beta} \int\limits_{0}^{t}
\left\| e^{-(t-s)\Delta_{\theta}} u^p(s)  \right\|_{L^q(\mathbb{R}^d_{\theta})}ds\nonumber\\
&\le &\alpha +
 Ct^{\beta}
\int\limits_{0}^{t} (t-s)^{- \frac{d}{2}\left(\frac{p}{q} - \frac{1}{q}\right)}
\|u^{p}(s)\|_{L^{\frac{q}{p}}(\mathbb{R}^d_{\theta})} ds \nonumber\\
&=& \alpha +
  Ct^{\beta}
\int\limits_{0}^{t}
(t-s)^{- \frac{d(p-1)}{2q}}
 \|u(s)\|_{L^q(\mathbb{R}^d_{\theta})}^{p}   ds \\
&\le &\alpha +
 Ct^{\beta} M^{p}
\int\limits_{0}^{t}
(t-s)^{ - \frac{d(p-1)}{2q}}
s^{-p\beta} ds\nonumber\\
&\overset{s=t\tau}{\leq}&\alpha +
 CM^{p}
\int\limits_{0}^{1}
(1-\tau)^{- \frac{d(p-1)}{2q}}
\tau^{-p\beta} d\tau,\nonumber
\end{eqnarray}
where $C:=(4\pi)^{- \frac{d(p-1)}{2q}}.$ The power of $t$ cancels by \eqref{4.15}. Using the Beta--Gamma identity,
$$
\int\limits_{0}^{1} (1-\tau)^{x-1} \tau^{y-1} d\tau
= \frac{\Gamma(x)\Gamma(y)}{\Gamma(x+y)}
$$ and \eqref{4.15},    we establish
\begin{equation}\label{eq:5.11}
 t^{\beta} \|\mathcal{A}(u)(t)\|_{L^{q}(\mathbb{R}^d_{\theta})} \overset{\eqref{4.15}}\le \alpha+C\frac{\Gamma((p-1)\beta) \Gamma(1 - p\beta)}{\Gamma(1 - \beta)}  M^{p}.
\end{equation}
Hence, if $\alpha$ and $M$ are chosen sufficiently small such that
$$
\alpha + M^{p} C\frac{\Gamma((p-1)\beta) \Gamma(1 - p\beta)}{\Gamma(1 - \beta)}\le M,
$$
then $\mathcal{A}(u)\in\Omega_\beta$ for every $u\in\Omega_\beta$; positivity follows from the positivity of the heat semigroup. Next, it follows from \eqref{nonlinearity_estimate} that
\begin{eqnarray}\label{eq:5.12}
\|u^{p}(s) - v^{p}(s)\|_{L^{q/p}(\mathbb{R}^d_{\theta})}&\overset{\eqref{nonlinearity_estimate}}\le& \left(\|u(s)\|_{L^q(\mathbb{R}^d_{\theta})}^{p-1} + \|v(s)\|_{L^q(\mathbb{R}^d_{\theta})}^{p-1}\right)
\|u(s) - v(s)\|_{L^{q}(\mathbb{R}^d_{\theta})}\nonumber\\
&\le&2M^{p-1}s^{-(p-1)\beta}\|u(s) - v(s)\|_{L^{q}(\mathbb{R}^d_{\theta})},\quad u,v\in\Omega_{\beta}.
\end{eqnarray}
A similar argument to that used in \eqref{Eq:5.32}, together with \eqref{eq:5.12}, shows that
\begin{eqnarray*}
t^{\beta} \|\mathcal{A}(u) - \mathcal{A}(v)\|_{L^{q}(\mathbb{R}^d_{\theta})}
&\le &
 C t^{\beta}
\int\limits_{0}^{t} (t-s)^{- \frac{d}{2}\left(\frac{p}{q} - \frac{1}{q}\right)}
\|u^{p}(s) - v^{p}(s)\|_{L^{\frac{q}{p}}(\mathbb{R}^d_{\theta})} ds \\
&\overset{\eqref{eq:5.12}}\leq&
 2CM^{p-1}t^{\beta}
\int\limits_{0}^{t}
(t-s)^{- \frac{d(p-1)}{2q}}s^{-(p-1)\beta}\|u(s) - v(s)\|_{L^{q}(\mathbb{R}^d_{\theta})} ds \\
&\leq&2CM^{p-1}t^\beta\|u-v\|_{\Omega_\beta}
\int_0^t(t-s)^{-\frac{d(p-1)}{2q}}s^{-p\beta}\,ds\\
&\leq &2CM^{p-1} \frac{\Gamma((p-1)\beta) \Gamma(1 - p\beta)}{\Gamma(1 - \beta)}\|u - v\|_{\Omega_{\beta}},
\end{eqnarray*}
where $C:=(4\pi)^{- \frac{d(p-1)}{2q}}.$  Hence, if
$$
2CM^{p-1}\frac{\Gamma((p-1)\beta)\Gamma(1-p\beta)}{\Gamma(1-\beta)}<1,
$$
then $\mathcal{A}$ is a contraction on $\Omega_{\beta}.$ Therefore, $\mathcal{A}$ has a unique fixed point $u\in\Omega_\beta$. At this stage we only know that $u\in C((0,\infty),L^q(\mathbb R^d_\theta))$; continuity at $t=0$ will follow below by comparison with the local solution.
In order to show that the fixed point $u$ of $\mathcal{A}$ is a mild solution to the problem \eqref{Main-equation1}-\eqref{Main-equation2}, we next establish that  $u \in C\left((0,\infty), L^\infty(\mathbb{R}^d_{\theta})\right).$ We first show that $u \in C\left((0, T], L^\infty(\mathbb{R}^d_{\theta})\right)$ for some sufficiently small $T>0$.
In fact, the previous argument guarantees uniqueness in the space $\Omega_{\beta,T}$, $T>0$, defined by
$$
\Omega_{\beta,T}
=
\left\{
u \in C\!\left((0, T), L^{q}(\mathbb{R}^{d}_{\theta})\right):
\sup_{0<t<T} t^{\beta}\|u(t)\|_{L^{q}(\mathbb{R}^{d}_{\theta})} \le M
\right\},
$$
where $\beta=\frac{1}{p-1}-\frac{d}{2q}$. Let $\tilde u$ be a local solution of problem \eqref{Main-equation1} -\eqref{Main-equation2} established in Theorem \ref{Local_existence}.  Note that
$$
u_0 \in L^{\infty}(\mathbb{R}^d_{\theta}) \cap L^{r}(\mathbb{R}^d_{\theta}) \subset L^{\infty}(\mathbb{R}^d_{\theta}) \cap L^{q}(\mathbb{R}^d_{\theta}).
$$
with $r < q < \infty.$
Thus, for $1< r < q < \infty,$ by  Theorem~\ref{Local_existence} there exists a maximal $T_{\max}>0$ depending on $u_0$ and
$$
\tilde u \in C\left((0,T_{\max}),  L^\infty(\mathbb{R}^d_{\theta}) \cap L^{q}(\mathbb{R}^d_{\theta})\right),
$$
such that $\tilde{u}=\mathcal{A}(\tilde{u}).$
Consequently, for sufficiently small \(T>0\), we have
$$
\sup\limits_{t\in[0, T]} t^{\beta} \|\tilde u(t)\|_{L^{q}(\mathbb{R}^d_{\theta})} \le M.
$$
Due to the uniqueness of solutions in $\Omega_{\beta,T},$ it follows that \(u = \tilde u\) on \([0, T]\),
leading to the conclusion that
$$
u \in C\left((0, T),  L^\infty(\mathbb{R}^d_{\theta}) \cap L^{q}(\mathbb{R}^d_{\theta})\right).
$$
To extend the regularity to the interval $[T,\infty)$, we use a bootstrap argument. For $t>T$, we write $u(t)$ as
$$
u(t) - e^{-t\Delta_{\theta}} u_0
= \int\limits_0^t e^{-(t-s)\Delta_{\theta}} u^p(s) ds=\int\limits_0^T e^{-(t-s)\Delta_{\theta}} u^p(s) ds
+ \int\limits_T^t e^{-(t-s)\Delta_{\theta}} u^p(s) ds=  I_{1}(t)+ I_{2}(t).
$$
Since $u \in C((0, T), L^{\infty}(\mathbb{R}^{d}_{\theta}) \cap   L^q(\mathbb{R}^{d}_{\theta}))$, by Lemma \ref{Lemma_1} and Lemma \ref{HK_Proposition1} \eqref{L_p_L_q_mapping_HK_prop_1} we obtain that
\begin{eqnarray*}
 \|I_1(t)\|_{L^{\infty}(\mathbb{R}^{d}_{\theta})}&=&\Big\|\int\limits_0^T\int\limits_{\mathbb{R}^d}G_{t-s}(\xi) T_{-\xi}u^p(s) d\xi ds \Big\|_{L^{\infty}(\mathbb{R}^{d}_{\theta})}\\
 &\le&\int\limits_0^T\int\limits_{\mathbb{R}^d}G_{t-s}(\xi)  \|T_{-\xi}u^p(s)\|_{L^{\infty}(\mathbb{R}^{d}_{\theta})}d\xi ds \\
 &\le&\sup_{0\leq s\leq T}\|u(s)\|^p_{C([0, T], L^\infty(\mathbb{R}^{d}_{\theta}))}\int\limits_0^T\|G_{t-s}\|_{L^{1}(\mathbb{R}^{d})} ds\\
 &\le&T\sup_{0< s< T}\|u(s)\|^p_{C((0, T), L^\infty(\mathbb{R}^{d}_{\theta}))}.
\end{eqnarray*}
That is, $I_1(t)\in L^\infty(\mathbb{R}^d_\theta)$ for all $t\geq T.$ It is not difficult to establish further that $I_{1} \in C([T,\infty), L^{\infty}(\mathbb{R}^{d}_{\theta})).$  Furthermore, one again the  similar argument to that used in \eqref{Eq:5.32} with $0<T<t,$ we obtain
\begin{eqnarray*}
 \|I_1(t)\|_{L^{q}(\mathbb{R}^d_{\theta})} \overset{\eqref{4.15}}\le  T^{-\beta}  C M^{p},
\end{eqnarray*}
From this another easy argument shows that $I_1 \in C\bigl([T,\infty), L^{q}(\mathbb{R}^d_{\theta})\bigr).$ Observe that \(q>r=\frac{d}{2}(p-1)\). Hence, one can choose \(r_1\in(q,\infty]\) such that
\begin{equation}\label{Eq:Parameters}
 \frac{d}{2}\left( \frac{p}{q} - \frac{1}{r_1}\right) < 1.
\end{equation}
For $t > T$, since $u \in C([0,\infty), L^{q}(\mathbb{R}^d_{\theta}))$, we have
$u^{p} \in C((T,t), L^{\frac{q}{p}}(\mathbb{R}^d_{\theta}))$. Thus, applying Lemma~\ref{HK_Proposition1} \eqref{L_p_L_q_mapping_HK_prop_1}   with \eqref{Eq:Parameters}, we deduce   $I_2 \in C([T,\infty), L^{r_1}(\mathbb{R}^d_{\theta})).$ Since the terms $e^{-t\Delta_{\theta}}u_0$ and $I_1$ belong to
$$
C([T,\infty), L^{\infty}(\mathbb{R}^d_{\theta})) \cap
C([T,\infty), L^{q}(\mathbb{R}^d_{\theta}))
\subset C([T,\infty), L^{r_1}(\mathbb{R}^d_{\theta})),
$$
we conclude that $u|_{[T,\infty)} \in C([T, \infty), L^{r_1}(\mathbb{R}^d_{\theta})).$  Next, choose a sequence $r_2,r_3,\ldots$ such that $r_{i+1}>r_i$ for $i=1,2,3,\ldots$ and
$$
\frac{d}{2}\left(\frac{p}{r_i}-\frac{1}{r_{i+1}}\right)<1, \quad i=2,3,\dots.
$$
eventually we have $r_i> \frac{dp}{2},$ and then we make take $r_{i+1}=\infty.$
Repeating the above procedure a finite number of times with $r_i, i=2,3,\dots,$   we conclude that
$$
u|_{[T,\infty)} \in C([T,\infty), L^\infty(\mathbb{R}^d_{\theta})).
$$
This completes the proof.
\end{proof}

\appendix
\section{Operator convexity}\label{ucp_appendix}
Let $\mathcal{M}$ be a von Neumann algebra. We denote by $M_n(\mathcal{M})$ the set of all $n\times n$ matrices
$a=[a_{ij}]_{i,j=1}^n$ with entries $a_{ij}\in \mathcal{M}.$
\begin{definition}\cite[Definition 3.3 p.194]{Takesaki}
Let $\mathcal{M}$ and $\mathcal{N}$ be von Neumann algebras. For each linear map $\Phi:\mathcal{M}\to \mathcal{N}$  we define a linear map $\Phi_n : M_n(\mathcal{M})\to M_n(\mathcal{N})$ by
$$
\Phi_n\big([a_{ij}]_{i,j=1}^n\big):=\big[\Phi(a_{ij})\big]_{i,j=1}^n.
$$
If $\Phi_n$ is positive, then $\Phi$ is called \emph{$n$-positive}.
If $\Phi$ is $n$-positive for all $n\in\mathbb{N}$, then $\Phi$ is called
\emph{completely positive}. The map $\Phi:\mathcal{M}\to \mathcal{N}$ is called \emph{unital} if it preserves the identity element, that is,
\[
\Phi(\mathbf{1}_{\mathcal M}) = \mathbf{1}_{\mathcal N}.
\]
Here, $\mathbf 1_{\mathcal{M}}$ and $\mathbf 1_{\mathcal{N}}$ are the identity operators in $\mathcal{M}$ and $\mathcal{N},$ respectively.
\end{definition}
The following result is well known, see e.g. \cite[Theorem 1.7 and Corollary 1.8. p.7]{Zhan}. We include a proof for completeness.
We will use the elementary fact that if $A \in \mathcal{M}$ is positive and invertible, and $B,C\in \mathcal{M},$ then
\begin{equation}\label{schur_complement_formula}
    \begin{pmatrix} A & B \\ B^* & C\end{pmatrix}\geq 0\;\Longleftrightarrow C\geq B^*A^{-1}B.
\end{equation}

\begin{prop}\label{general_vna_fact}
  Let $\mathcal{M}$ be a von Neumann algebra, and let $\Phi:\mathcal{M}\to \mathcal{M}$ be a unital completely positive map. For $1\leq p\leq 2,$ we have
  \begin{equation}\label{inequality_1}
        \Phi(u^p) \geq \Phi(u)^p,\quad 0\leq u\in \mathcal{M}.
  \end{equation}
\end{prop}
\begin{proof}
If $p=1$ there is nothing to prove. For $p=2$ and $ 0\leq u\in \mathcal{M},$ the matrix
$$
        \begin{pmatrix} u^2 & u \\ u & \mathbf{1}_{\mathcal M}\end{pmatrix}
$$
    is positive definite in $\mathcal{M}\otimes M_2(\mathbb{C}).$
    Here, $\mathbf{1}_{\mathcal M}$ is the identity element of $\mathcal{M}.$ Therefore,
$$
        \begin{pmatrix} \Phi(u^2) & \Phi(u) \\ \Phi(u) & \mathbf{1}_{\mathcal M}\end{pmatrix}\geq 0, \,\  0\leq u\in \mathcal{M}.
$$
    It follows from \eqref{schur_complement_formula} that $\Phi(u^2)\geq \Phi(u)^2.$ This is the $p=2$ case. Assume now that $1<p<2.$

    Let $t>0.$ The matrix
$$
        \begin{pmatrix} t+u & \mathbf{1}_{\mathcal M} \\ \mathbf{1}_{\mathcal M} & (t\cdot \mathbf{1}_{\mathcal M}+u)^{-1}\end{pmatrix},
$$
    is positive in $\mathcal{M}\otimes M_2(\mathbb{C}).$
    Since $\Phi$ is a unital completely positive map we have
$$
        \begin{pmatrix} \Phi(t\cdot\mathbf{1}_{\mathcal M}+u) & \mathbf{1}_{\mathcal M} \\ \mathbf{1}_{\mathcal M} & \Phi ((t\cdot\mathbf{1}_{\mathcal M}+u)^{-1}) \end{pmatrix}\geq 0, \,\ 0\leq u\in \mathcal{M},
$$
    in $\mathcal{M}\otimes M_2(\mathbb{C}).$ Therefore, again using \eqref{schur_complement_formula}, we have
$$
        \Phi((t\cdot\mathbf{1}_{\mathcal M}+u)^{-1}) \geq \Phi(t\cdot\mathbf{1}_{\mathcal M}+u)^{-1} = (t\cdot\mathbf{1}_{\mathcal M}+\Phi(u))^{-1}, \,\  0\leq u\in \mathcal{M},
$$
    where in the last equality we used the fact that $\Phi(t\cdot\mathbf{1}_{\mathcal M}+u)=t\cdot\mathbf{1}_{\mathcal M}+\Phi(u)$ for any $t>0$ and $0\leq u\in \mathcal{M}.$
    Since for any operator $A$ we have
$$
        A^2(t\cdot\mathbf{1}_{\mathcal M}+A)^{-1} = A-t\cdot\mathbf{1}_{\mathcal M}+t^2(t\cdot\mathbf{1}_{\mathcal M}+A)^{-1},
$$
    it follows that
$$
        \Phi(u^2(t\cdot\mathbf{1}_{\mathcal M}+u)^{-1})\geq \Phi(u)^2(t\cdot\mathbf{1}_{\mathcal M}+\Phi(u))^{-1}, \,\  0\leq u\in \mathcal{M}.
$$
    Using the integral formula
$$
        x^{p} = \frac{|\sin(p\pi)|}{\pi}\int_0^\infty t^{1-p}\frac{x^2}{t+x}\,dt,\quad 1<p< 2, \,\  0\leq x\in \mathcal{M},
$$
    we conclude
$$
        \Phi(u^p) \geq \Phi(u)^p, \,\  0\leq u\in \mathcal{M}.
$$
    This completes the proof.
\end{proof}
Now, let $\mathcal{M} = L^{\infty}(\mathbb{R}^d_\theta),$ and let $\Delta_{\theta}$ be the Laplacian on $L^{\infty}(\mathbb{R}^d_\theta).$
\begin{cor}\label{heat-cp-map}
The map
$$
\Phi_t(u) = e^{-t\Delta_{\theta}}(u),\quad u\in L^{\infty}(\mathbb{R}^d_\theta)
$$
is a unital completely positive map.
\end{cor}
\begin{proof}
    We have $T_{\eta}I = I,$ so
    $$
        e^{-t\Delta_{\theta}}(I) = \int\limits_{\mathbb{R}^d} G_t(\eta)I\,d\eta = I,
    $$
where $I$ is the identity operator in $L^{\infty}(\mathbb{R}^d_{\theta}).$ Note that if $u\geq 0,$ then $T_{\eta}u\geq 0$ for all $\eta\in \mathbb{R}^d.$ Hence,
    $$
        e^{-t\Delta_{\theta}}(u)\geq 0.
    $$
    The same holds for $T_{\eta}\otimes I_{M_N(\mathbb{C})}$ for any $N\geq 1.$ Since $\Phi_t(u) = G_t\ast u$ is a closed convex combination of translation maps, it follows that $\Phi$ is a positive unital linear map on $L^\infty(\mathbb{R}^d_\theta).$
\end{proof}

The following Lemma \ref{Jensen-inequality} can be derived as direct consequences of Proposition \ref{heat-cp-map} and the inequality \eqref{inequality_1} which play a key role in the proof of Lemma \ref{Lemma5_5}.
\begin{lem}\label{Jensen-inequality} Assume that $0\leq u\in L^\infty(\mathbb{R}^d_\theta)$ and $1<p\leq2$. Then, for $t>0$ we have
\begin{equation}\label{inequality_4}
 \left(e^{-t\Delta_{\theta}}(u)\right)^p\le e^{-t\Delta_{\theta}}(u^p) , \quad t>0.
\end{equation}
\end{lem}
\begin{proof}  Combining Proposition \ref{general_vna_fact} with Corollary~\ref{heat-cp-map}, we derive the inequality \eqref{inequality_4}.
 \end{proof}

{\bf Conflict of Interest Statement:} The author declares that there is no conflict of interest.

{\bf Data Availability Statement:} No datasets were generated or analyzed during the current study.

\section*{Acknowledgements}
We did not use any AI tools in developing the mathematical results of this paper. Indeed, an earlier version of this work was circulated as the preprint arXiv:2601.16053 in January, 2026 before AI tools reached their current level of development. Only minor revisions were made in preparing the present manuscript. The authors are responsible for verifying all arguments and references and for the final content of the manuscript.

K.T. was partially supported by the Australian Research Council (ARC) grant FL17010005. M.R. was supported by the EPSRC grant UKRI3645, and by the Methusalem grant (01M01021 BOF Methusalem). The  third and fourth authors would like to thank to Professor Berikbol Torebek for the helpful discussion on the classical Fujita's result. E.M. was supported by the Agence Nationale de la Recherche (ANR) project OpART (Operator Algebras and Representation Theory) and the Alexander von Humboldt Foundation. Thanks are due to Deyu Chen, who pointed out a critical error in an earlier version of this paper.

\begin{center}

\end{center}

\end{document}